\pdfoutput=1
\documentclass[reqno]{amsart}

\usepackage{amsmath}
\usepackage{amsthm}
\usepackage{amssymb}
\usepackage{dsfont}
\usepackage{hyperref}
\usepackage{mathtools}
\usepackage{bm}
\usepackage{xcolor}
\usepackage{booktabs}
\usepackage{tikz}
\usepackage{pgfplots}
\pgfplotsset{compat=1.18}
\definecolor{dark2Green}{HTML}{1B9E77}
\definecolor{dark2Orange}{HTML}{D95F02}
\definecolor{dark2Purple}{HTML}{7570B3}
\definecolor{dark2Pink}{HTML}{E7298A}
\usetikzlibrary{pgfplots.colorbrewer, backgrounds, arrows.meta, positioning}
\usepackage{cleveref}
\crefrangeformat{subsection}{subsections~#3#1#4--#5#2#6}
\crefname{subsection}{subsection}{subsections}

\hypersetup{colorlinks=true,linkcolor=dark2Pink, citecolor=dark2Green, urlcolor=dark2Purple}

\newtheorem{theorem}{Theorem}
\newtheorem{lemma}[theorem]{Lemma}
\newtheorem{corollary}[theorem]{Corollary}
\newtheorem{proposition}[theorem]{Proposition}
\theoremstyle{definition}
\newtheorem{remark}[theorem]{Remark}
\newtheorem{example}[theorem]{Example}
\newtheorem{definition}[theorem]{Definition}

\newcommand{\R}{\mathbb R}
\newcommand{\C}{\mathbb C}
\newcommand{\N}{\mathbb N}

\newcommand{\B}{\mathcal{B}}
\newcommand{\I}{\mathcal{I}}

\newcommand{\diff}{\mathrm{d}}
\newcommand{\derb}{\left(\frac{\diff}{\diff t}\right)}
\newcommand{\tderb}{(\tfrac{\diff}{\diff t})}

\newcommand{\lag}{\mathfrak{l}}
\newcommand{\n}{\mathfrak{n}}

\newcommand{\col}{\operatorname{col}}
\newcommand{\rank}{\operatorname{rank}}
\newcommand{\im}{\operatorname{im}}

\DeclarePairedDelimiter\norm{\lVert}{\rVert}
\DeclarePairedDelimiter\scp{\langle}{\rangle}
\DeclarePairedDelimiter\abs{\lvert}{\rvert}

\title[A continuous-time fundamental lemma]{
A continuous-time fundamental lemma\\ and its application in data-driven optimal control
}

\author[P.~Schmitz]{Philipp Schmitz}
\address[P.~Schmitz]{Technische Universität Ilmenau, Institute of Mathematics, Optimization-based Control Group, Ilmenau, Germany}
\email{philipp.schmitz@tu-ilmenau.de}

\author[T.~Faulwasser]{Timm Faulwasser}
\address[T.~Faulwasser]{Hamburg University of Technology, Institute of Control Systems, Hamburg, Germany}%
\email{timm.faulwasser@ieee.org}

\author[P.~Rapisarda]{Paolo Rapisarda}
\address[P.~Rapisarda]{University of Southampton, Southampton, United Kingdom}%
\email{pr3@ecs.soton.ac.uk}

\author[K.~Worthmann]{Karl Worthmann}
\address[K.~Worthmann]{Technische Universität Ilmenau, Institute of Mathematics, Optimization-based Control Group, Ilmenau, Germany}%
\email{karl.worthmann@tu-ilmenau.de}

\thanks{P.~Schmitz is grateful for the support from the Carl Zeiss Foundation (VerneDCt -- Project No.\ 2011640173). K.~Worthmann gratefully acknowledges funding by the German Research Foundation (DFG; project number 507037103).}%

\begin{document}

\begin{abstract}
Data-driven control of discrete-time and continuous-time systems is of tremendous research interest. In this paper, we
explore data-driven optimal control of continuous-time linear systems using input-output data. Based on a density result, we rigorously derive error bounds for finite-order polynomial approximations of elements of the system behavior. To this end, we leverage a link between latent variables and flat outputs of controllable systems. Combined with a continuous-time counterpart of the fundamental lemma by Willems et al., we characterize the suboptimality resulting from  polynomial approximations in data-driven linear-quadratic optimal control problems. Finally, we draw upon a numerical example to illustrate our results.
\end{abstract}
\smallskip

\maketitle

\section{Introduction}

Data-driven control, i.e., the design of controller and feedback laws directly from measured data, is a topic of ongoing research interest, see~\cite{DataCon_CSMI,DataCon_CSMII} and the range of articles in these special issues. At the core of many developments for linear discrete-time systems is a pivotal result, known as Willems et al.'s \textit{fundamental lemma}~\cite{willems2005note}, which allows to parameterize the external (input-output) finite-horizon behavior of controllable systems from sufficiently informative data, see, e.g., the recent survey~\cite{FaulOu23} and the references therein. The interest in this result has been catalyzed by recent works such as \cite{coulson2019data,de2019formulas,markovsky2021behavioral}; applications of data-driven control concepts are discussed in the literature \cite{schmitz2022data,bilgic2022toward,markovsky2023data,verhoek2023direct}. 
In \cite{Mueller22,Lopez24} and \cite{Rapisarda23b} continuous-time extensions of the fundamental lemma have been proposed. While the former works require to solve a scalar ODE to compute future trajectories, in the latter paper an approach  to compute a generating representation based on polynomial series expansions of input-output trajectories is proposed. 
However, the latter result has not yet been used to design data-driven controllers. 
In the present paper, we address this gap: 
We extend the results from~\cite{Rapisarda23b} by deriving errors bounds for polynomial series expansions on elements of the behavior. 
We also prove that the set of polynomial system trajectories is dense in the set of all system trajectories. 
We use the approximation bounds to derive bounds on the optimality gap resulting from using finite-order polynomials to solve linear–quadratic regulator (LQR) problem formulated in terms of the behavior. 
Finally, we establish a continuous-time fundamental lemma involving the Gramians of trajectories and, based on this, present a data-driven approximation for the LQR problem.

The remainder of this paper is structured as follows:
In \Cref{sec:Intro} we revisit foundational concepts in behavioral systems theory, with most results rederived to suit our specific setting. Utilizing the connection between flat outputs and latent variables, in \Cref{lem:flatimage}, we derive a specific behavioral representation, which is beneficial for subsequent approximation results.
Moreover, we recap properties of Legendre polynomials (\Cref{subsec:polylift,subsec:trunc}) and analyze their advantage in the approximation of behavioral elements (\Cref{subsec:polytraj}) as shown in \Cref{lem:poly_traj1}.
In \Cref{sec:LQR_approx}, we introduce a version of the finite-horizon linear-quadratic optimal control problem and its finite-dimensional approximation in the space of linear combinations of Legendre polynomials including a convergence analysis. 
In \Cref{sec:f-lemma}, we define the concept of persistency of excitation (\Cref{sec:PE}) and we state a continuous-time fundamental lemma (\Cref{thm:FL}). \Cref{subsec:ident} demonstrates how the fundamental lemma can be applied in system identification.
In \Cref{sec:dataOptCon}, we discuss the application of our results to the data-driven solution of the finite-horizon optimal control problem before we conclude the paper in \Cref{Sec:Conclusion} pointing out directions of current and future research. \medskip

\noindent\textbf{Notation}:
Given two sets $X$ and $\Omega$, the set of functions $f:\Omega\rightarrow X$ is denoted by $X^\Omega$. Let $\I$ be a real interval; we denote by $\overline{\I}$ the closure of $\I$. Let $d, k\in\N$; then $L^2(\I, \R^d)$ denotes the space of equivalence classes of square integrable functions $f\in(\R^d)^\I$ and $H^k(\I,\R^d)$ is the $k$th order Sobolev space associated with $L^2(\I, \R^d)$. The scalar product in $L^2(\I,\R^d)$ and its induced norm are given by $\scp{f,g}=\int_{\I} f(\tau)^\top g(\tau)\,\mathrm d\tau$ and $\norm{f}=\sqrt{\scp{f,f}}$. The usual norm in $H^k(\I,\R^d)$ is denoted by $\norm{\,\cdot\,}_{H^k}$. For $k\in\N\setminus\{0\}$ and $f\in H^{k-1}(\I, \R^d)$ we set \begin{equation}
\label{eq:Lambda}
    \Lambda_{k}(f) = \begin{bmatrix}
    f \\\vdots\\ f^{(k-1)}
\end{bmatrix}\in L^2(\I, \R^{k\cdot d}).
\end{equation} In particular, $\norm{f}_{H^{k-1}} = \norm{\Lambda_{k}(f)}$. $\mathcal C^\infty(\I, \R^d)$ is the space of infinitely differentiable functions from $\I$ to $\R^d$ and $\mathcal C_\mathrm{c}^\infty(\I, \R^d)$ consists of those functions of $\mathcal C^\infty(\I, \R^d)$ with compact support. Given a Hilbert space $X$, $\ell^2(\N,X)$ is the space of square summable sequences in $X^\N$.

The identity operator from a vector space $X$ onto itself is denoted by $I_X$ or simply $I$ when clear from the context. In the case of a finite dimensional space $X=\R^d$ we also write $I_d$. The Euclidean norm in $\R^d$ is denoted by $\norm{\,\cdot\,}_2$. If $A_0,\dots, A_k$ are matrices with the same number of columns, we define $\col(A_0,\dots, A_k):=\begin{bmatrix}
  A_0^\top&\dots &  A_k^\top
\end{bmatrix}{}^\top$. We always identify $\R^d$ with $\R^{d\times 1}$. Given a matrix $M$, we denote by $\im M$ and $\ker M$ its image and kernel. Further, $M^\top$ and $M^\dagger$ denote the transpose and Moore--Penrose inverse.

Finally, we denote by $\R[s]$ the ring of polynomials with real coefficients in the indeterminate $s$, and by $\R^{g\times q}[s]$ the ring of $g\times q$ polynomial matrices with real coefficients. 

\section{Linear differential systems}\label{sec:Intro}
We first recapitulate behavioral concepts for linear time-invariant systems. We then connect this paradigm to polynomial series expansions of trajectories and explore how polynomial trajectories can approximate system behaviors.

\subsection{Behaviors}\label{subsec:behaviors}
In the following, we deal with dynamical systems given by the time interval $\I=(-1,1)$\footnote{We choose such interval purely for simplicity of notation; with straightforward modifications, \emph{any} other bounded open interval can be used.}, the signal space $\R^q$ and a \emph{behavior} $\B\subset (\R^q)^\I$. We focus on \emph{linear differential behaviors}, i.e.\ the set of solutions to a system of linear, constant-coefficient differential equations
\begin{equation}\label{eq:ker}
    R\left(\frac{\diff}{\diff t}\right)w=0,
\end{equation}
where $R(s)=R_0 s^0 + \dots + R_r s^r$ is a polynomial matrix in $\R^{g\times q}[s]$. Here, the solution $w$ of \eqref{eq:ker} is meant in the sense of \emph{weak solutions}, i.e.,\ $w\in L^2(\I,\R^q)$ and it satisfies
\begin{equation}
\label{eq:weak}
    0=\sum_{i=0}^r (-1)^i\int_\I w^\top R_i^\top \tfrac{\diff^i \phi}{\diff t^i}\,\diff t
\end{equation}
for all test functions $\phi\in \mathcal C_\mathrm{c}^\infty(\I, \R^{g})$. Given $w\in \mathcal C^\infty(\overline{\I},\R^q)$ integration by parts shows that \eqref{eq:weak} is valid if and only \eqref{eq:ker} holds pointwise. 

Given our choice of solution set, we need to slightly generalize and prove some well-known results from \cite{yellowbook}, where the solutions of \eqref{eq:ker} are assumed to be infinitely differentiable. In particular, the equivalence of the different representations of behaviors, established essentially for smooth functions in \cite{yellowbook}, requires verification in the context of weak $L^2$-solutions.
\begin{lemma}
    \label{lem:behavclosed}
    The behavior
    \begin{equation}
        \B := \{w\in L^2(\I,\R^q)\,|\, R(\tfrac{\diff}{\diff t})w=0\}
    \end{equation}
    is closed in $L^2(\I,\R^q)$.
\end{lemma}
\begin{proof}
    Consider a sequence $(w_n)_{n\in\N}$ in $\B$ which converges to some $w\in L^2(\I,\R^q)$. Note that for each $\phi\in\mathcal C_\mathrm{c}^\infty(\I, \R^{g})$ the right hand side in \eqref{eq:weak} defines a linear, continuous functional $f_\phi:L^2(\I,\R^q)\rightarrow \R$. By continuity $0=\lim_{n\rightarrow \infty} f_\phi(w_n)=f_\phi(w)$ for every $\phi\in\mathcal C_\mathrm{c}^\infty(\I, \R^{g})$, that is $w$ solves \eqref{eq:ker} and $w\in\B$.
\end{proof}

We recall the notion of behavioral controllability, cf.\ Definition~5.2.2 in \cite{yellowbook}.
    The behavior $\B$ is \emph{controllable} if for each two trajectories $w_0$,~$w_1\in \B$ there is $t_1\in (0,1)$ and $w\in\B$ such that
    \begin{equation}
        \label{eq:patchingup}
        w(t) = \begin{cases}
            w_0(t) & \text{if }t\in (-1,0],\\
            w_1(t-t_1) & \text{if } t\in [0, 1).
        \end{cases}
    \end{equation}

Partitioning $R$ compatibly with a known selection of inputs and outputs  $w=\col(u,y)\in\B$, cf.\ Definition~3.3.1 in \cite{yellowbook}, one obtains the \emph{input-output representation}
\begin{equation}\label{eq:io}
    P\derb y = Q\derb u,
\end{equation}
where $P\in \R^{p\times p}[s]$ can be assumed to be nonsingular and $Q\in \R^{p\times m}[s]$ with $q=m+p$.

A linear differential behavior $\B$ also admits an \emph{input-state-output representation} (see \cite{Rap97}), i.e.,\ there are matrices $A\in\R^{n\times n}$, $B\in\R^{n\times n}$, $C\in\R^{p\times n}$, $D\in\R^{p\times m}$ such that for all $w=\col(u,y)\in\B\cap \mathcal C^\infty(\overline{\I},\R^q)$ there exists $x\in \mathcal C^\infty(\overline{\I},\R^n)$ with
\begin{equation}\label{eq:sys}
\begin{split}
    \frac{\diff}{\diff t} x &= Ax + Bu\\
    y &= Cx+Du.
\end{split}
\end{equation}

The following result relates the \emph{external behavior} described by \eqref{eq:io}, consisting of all input-output trajectories, and its input-state-output counterpart in the terms of closed $L^2$-subspaces. 
\begin{lemma}
    \label{lem:sys}
    Given an input-state-output representation~\eqref{eq:sys} of $\B$, one has
    \begin{equation}
    \label{eq:isobehav}
        \mathcal B = \left\{\col(u,y)\in L^2(\I,\R^{m+p})\,\middle|\, \begin{aligned} &\exists~x\in H^1(\I,\R^n)\ \\&\text{s.t. \eqref{eq:sys} holds}\end{aligned}\right\}.
    \end{equation}
\end{lemma}
\begin{proof}
    Denote the set on the right hand side of~\eqref{eq:isobehav} by $\widetilde{\B}$. We show that $\widetilde{\B}$ is a closed subspace of $L^2(\I,\R^{m+p})$. To this end consider the solution operator $S:L^2(\I, \R^m)\rightarrow H^1(\I,\R^n)$ defined by 
    \begin{equation}
        (Su)(t) := \int_{-1}^t \exp(A(t-\tau)) B u(\tau)\,\mathrm d\tau.
    \end{equation}
  From the definition of $\widetilde{\B}$ it follows that $\col(u,y)\in\widetilde{\B}$ if and only if there exists $x^0\in\R^d$ such that $x=\exp(A(\cdot+1))x^0 + Su$ with $y=Cx + Du$. Therefore, $\widetilde{\B}$ is the direct sum of the finite-dimensional space
  \begin{equation*}
    \widetilde{\B}_0 := \{\col(0,C\exp(A(\cdot+1))x^0)\,|\, x^0\in\R^n\}
  \end{equation*}
  and
    \begin{equation}
        \widetilde{\B}_1 := \{\col(u,(CS+D)u)\,|\, u\in L^2(\I,\R^m)\}.
    \end{equation}
    The space $\widetilde{\B}_1$ is closed in $L^2(\I,\R^{m+p})$ as it is the graph of the bounded linear operator $(CS+D):L^2(\I,\R^{m})\rightarrow L^2(\I,\R^{p})$. This shows the closedness.

    The assertion follows with $\B\cap \mathcal C^\infty(\I,\R^{m+p})=\widetilde{\B}\cap \mathcal C^\infty(\I,\R^{m+p})$ and a density argument. \qedhere
\end{proof}

\begin{remark}
    Obtaining a kernel representation \eqref{eq:ker} from an input-state-output one can be achieved by \emph{elimination} of the state variable $x$, see Section 6.2.2 of \cite{yellowbook}.
\end{remark}

In this paper we use a couple of integer system invariants. The \emph{McMillan degree} of $\B$, denoted $\n(\B)$, is the minimal dimension of the state space among all possible input-state-output representations~\eqref{eq:sys} of $\B$. If the state space dimension equals $\n(\B)$, this particular input-state-output representation is said to be 
\emph{minimal}. Define 
\begin{equation}\label{eq:Ok}
\mathcal{O}_k:=\begin{cases}
C&\mbox{ if } k=0\\
\begin{bmatrix} \mathcal{O}_{k-1}\\CA^{k}\end{bmatrix}&\mbox{ if } k\geq 1 
\end{cases}\; ;
\end{equation}
the \emph{system lag}, denoted $\mathfrak{l}(\B)$, is defined by 
\begin{equation*}
\mathfrak{l}(\B):=\min \{k\in\N \mid \rank~ \mathcal{O}_k=\rank~ \mathcal{O}_{k-1} \}\; .
\end{equation*}
Evidently, $\mathfrak{l}(\B)\leq \n(\B)$. Further, $\mathfrak{l}(\B)$ is the highest order of differentiation in a ``shortest lag" description of $\B$, see pp. 569-570 of \cite{Willems86}.

Given an input-state-output representation of $\B$ the state variable $x$ is called \emph{observable}, if it can be recovered from the input-output trajectory, i.e.,\ there is a polynomial matrix $F\in\R^{n\times q}[s]$ such that for all $w=\col(u,y)\in \B\cap \mathcal C^\infty(\overline{\I},\R^q)$
\begin{equation}
    \label{eq:observ}
    x = F\derb w.
\end{equation}
Observability of the state variable is equivalent to $(A,C)$ being observable in the usual sense (see e.g.\ \cite{Trentelman2001}), which is satisfied if $\mathcal O_{(n-1)} = n$. For an observable pair $(A,C)$, the {lag} $\lag(\B)$ is the observability index of the pair.

In a manner akin to observability, a connection  between behavioral controllability and controllability of the pair $(A, B)$ can be established in terms of input-state-output representations, see, e.g.,\ \cite{Trentelman2001}.

\begin{lemma}
    Suppose $\B$ is controllable and consider a minimal input-state-output representation~\eqref{eq:sys} of $\B$.
    Then $(A,B)$ is controllable and $(A,C)$ is observable.
\end{lemma}
\begin{proof}
   An input-state-output representation of $\B$ is minimal if and only if $(A,C)$ is observable the input-state-output representation is \emph{state trim}, i.e.,\ for all $x^0\in\R^n$ there is $\col(u,y)\in\B\cap \mathcal C^\infty(\overline{\I},\R^q)$ and $x\in \mathcal C^\infty(\overline{\I},\R^n)$ such that \eqref{eq:sys} and $x(0)=x^0$ hold, cf.\ \cite{Rap97}. We only need to show the controllability of $(A,B)$, that is for arbitrary initial value $x^0\in\R^n$ and terminal value $x^1\in\R^n$ there is a control input $u$ and a time instance $t_1\in(0,1)$ such that the state solution $x$ of $\tfrac{\diff}{\diff t}x = Ax+Bu$ satisfies $x(0)=x^0$ and $x(t_1)=x^1$. By state trimness we find $w_0=\col(u_0,y_0)$,~$w_1=\col(u_1,y_1)\in\B\cap \mathcal C^\infty(\overline{\I},\R^q)$ and corresponding states $x_0, x_1\in \mathcal C^\infty(\overline{\I},\R^n)$ with $x_0(0)=x^0$ and $x_1(0)=x^1$. Since $\B$ is controllable, there further exists $w=\col(u,y)\in\B\cap\mathcal C^\infty(\overline{\I},\R^n)$
   and $t_1\in (0,1)$ satisfying~\eqref{eq:patchingup}, cf.\ Theorem~5.2.9 in \cite{yellowbook}. 
   With the observability of the state~\eqref{eq:observ} we see
   \begin{align*}
    x(0) &= \left(F\tderb w\right)(0) = \left(F\tderb w_0\right)(0) = x(0)=x^0,\\
    x(t_1) &= \left(F\tderb w\right)(t_1) = \left(F\tderb w_1\right)(0) = x_1(0) = x^1.\qedhere
   \end{align*}
\end{proof}

For a linear differential behavior controllability is equivalent to the existence of an \emph{image representation} (see Theorem 6.6.1 p.\ 229 \emph{ibid.}), i.e.,\ there exists a polynomial matrix $M\in\R^{(m+p)\times m}[s]$ such that
$w\in\B\cap \mathcal C^\infty(\overline{\I},\R^{m+p})$ if and only if there exists a \emph{latent variable} trajectory $\ell\in\mathcal C^\infty(\overline{\I},\R^{m})$ such that 
\begin{equation}\label{eq:image}
   w=M\derb \ell.
\end{equation}

A flat output of a differentially flat system leads to a specific image representation endowed with helpful characteristics.
\begin{lemma}
    \label{lem:flatimage}
    Suppose that $\B$ is controllable. Then there exists $M\in\R^{q\times m}[s]$ with $\deg(M)\leq \n(\B)+1$ such that 
    \begin{equation}
        \label{eq:imagebehav}
        \mathcal B = \left\{\col(u,y)\in L^2(\I,\R^{m+p})\,\middle|\, \begin{aligned} &\exists~\ell\in L^2(\I,\R^m)\ \\&\text{s.t. \eqref{eq:image} holds}\end{aligned}\right\}.
    \end{equation}
    Moreover, given $w\in H^k(\I,\R^q)$ for some $k\in\N$ the latent variable in \eqref{eq:image} satisfies $\ell\in H^{k+1}(\I,\R^m)$.
\end{lemma}

\begin{proof}
    We consider a minimal input-state-output representation~\eqref{eq:sys} of $\B$, that is $n=\n(\B)$ and  $(A,B)$ is controllable. Then there is a \emph{flat output} defined by
    \begin{equation}
        \label{eq:flatoutput}
        \ell = \widetilde C x
    \end{equation} with output matrix $\widetilde C\in\R^{m\times n}$, see pp.\ 84--ff.\ in \cite{SiraRamirez2018} and Remark~2 p.\ 72 of \cite{Levine2003}. In more detail, there are polynomial matrices $X\in\R^{n\times m}[s]$, $U\in\R^{m\times m}[s]$ with $\deg(X)\leq n$ and $\deg(U)\leq n+1$ such that, given $w=\col(u,y)\in\B\cap \mathcal C^\infty(\overline{\I},\R^q)$ with corresponding state $x\in\mathcal C^\infty(\overline{\I},\R^n)$,
    \begin{equation*}
        x = X\tderb\ell,\quad u=U\tderb\ell, \quad y= (CX\tderb + DU\tderb)\ell.
    \end{equation*}
    The associated image representation~\eqref{eq:image} is established through
    \begin{equation}
        \label{eq:image_flat}
        M = \begin{bmatrix}
            U\\
            CX + DU
        \end{bmatrix}
    \end{equation}
    and the  flat output $\ell$ serves as latent variable.

    We derive~\eqref{eq:imagebehav} for this particular $M$. Denote the set on the right side of~\eqref{eq:imagebehav} by $\widetilde{\B}$. We show that $\widetilde{\B}$ is closed in $L^2(\I,\R^{q})$. Let $(\col(u_k,y_k))_{k\in\N}$ be a sequence in $\widetilde{\B}$ which converges in $L^2(\I,\R^{q})$ to some $\col(u,y)$. The  state and latent variable corresponding to $\col(u_k,y_k)$ are denoted by $x_k$ and $\ell_k$, respectively. As the latent variable is given via a flat output (see \eqref{eq:flatoutput}) we find
    \begin{equation*}
        \ell_k = \widetilde C x_k = \widetilde C Su_k
    \end{equation*}
    where $S$ is the solution operator defined in the proof of \Cref{lem:sys}. The convergence of $(u_k)_{k\in\N}$ and the boundedness of $S$ imply that $(\ell_k)_{k\in\N}$ converges to some $\ell\in L^2(\I,\R^{m})$. With \Cref{lem:behavclosed} we know that
    \begin{equation*}
        \B_\ell =\{\col(u,y,\ell)\in L^2(\I,\R^{q+m})\,|\, \text{\eqref{eq:image} holds}\}
    \end{equation*}
    is closed in $L^2(\I,\R^{q+m})$. Therefore, $\col(u,y,\ell)\in\B_\ell$ and $\col(u,y)\in \widetilde{\B}$. Since $\widetilde \B\cap \mathcal C^\infty(\overline{\I},\R^{q})$ and $\B\cap \mathcal C^\infty(\overline{\I},\R^{q})$ coincide, a density argument yields~\eqref{eq:imagebehav}.

    Moreover, if $\col(u,y)\in H^k(\I,\R^q)$, then the corresponding state satisfies $x=Su\in H^{k+1}(\I,\R^n)$ and, thus, $\ell=\widetilde Cx \in H^{k+1}(\I,\R^m)$.
\end{proof}

\subsection{Polynomial lift}\label{subsec:polylift}

Let $(\pi_i)_{i\in\R}$ be the sequence of Legendre polynomials with normalization $\pi_i(1)=1$, $i\in\N$, which forms an orthogonal basis of $L^2(\I,\R)$. Recall that $\pi_0(t)=1$, $\pi_1(t)=t$ and for $i=1,\ldots$ 
\[
\pi_{i+1}(t)=\frac{2i+1}{i+1}t \pi_i(t)-i\pi_{i-1}(t)\; .
\]
Given $f\in L^2(\I,\R)$ there is a unique series expansion 
\begin{equation}
    f = \sum_{i\in\N} \hat f_i \pi_i,
\end{equation}
where $\hat f_i:=\scp{f,\pi_i} \norm{\pi_i}^{-2}$ and $\hat f:= (\hat f_i)_{i\in\N}$. From Bessel's theorem it follows that $\hat f\in \ell^2(\N,\R)$.

Any function $f\in L^2(\I,\R^d)$ with dimension $d\geq 1$ can be represented with coefficients defined by
\begin{equation}
    \hat f_i := \sum_{k=0}^{d-1} \frac{\scp{f,e_k\pi_i}}{\norm{\pi_i}^2} e_k,
\end{equation}
where $\{e_0,\dots e_{d-1}\}$ is the canonical basis of $\R^d$ and $\hat f= (f_i)_{i\in\N}\in \ell^2(\N,\R^d)$. We define 
\begin{equation}
    \Pi:L^2(\I,\R^d)\rightarrow \ell^2(\N,\R^d),\quad f\mapsto \hat f,
\end{equation} which is an isometric isomorphism.

The differential operator $\frac{\mathrm d}{\mathrm dt}$ on $H^1(\I,\R^d)$ can be represented as an operator $\mathcal D$ acting in $\ell^2(\N,\R^d)$,~\cite[Equation~(2.3.18)]{canuto2007spectral}. For functions $f\in H^1(\I,\R^d)$ with $\hat f = \Pi f$ and $\widehat{f^{(1)}}:=\Pi \frac{\mathrm d f}{\mathrm dt}$ one has
\begin{equation}
    \label{eq:infMat}
    (\widehat{f^{(1)}})_i = (\mathcal D \hat f)_i := (2i+1)\sum_{\substack{j=i+1\\i+j\text{ odd}}}^\infty \hat f_j,\quad i\in\N,
\end{equation}
or equivalently written by means of an infinite matrix
\begin{equation}
    \label{eq:Dmat}
    \begin{bmatrix}
        (\widehat{f^{(1)}})_0\\
        (\widehat{f^{(1)}})_1\\
        (\widehat{f^{(1)}})_2\\
        \vdots
    \end{bmatrix}= \begin{bmatrix}
         0 & I & 0 & I &0 & I & \dots\\
         & 0 & 3I & 0 & 3I & 0 & \dots\\
         && 0 & 5I & 0 & 5I & \dots\\
         &&& 0 & 7I & 0 & \dots\\
         &&&& \ddots & \ddots & \ddots
     \end{bmatrix}\begin{bmatrix}
         \hat f_0\\\hat f_1\\\hat f_2\\\vdots
     \end{bmatrix}.
\end{equation}
Similarly to the differential operator $\tfrac{\mathrm d}{\mathrm dt}$ one can define powers and polynomials of $\mathcal D$.

Employing the kernel representation~\eqref{eq:ker} we find the following characterization of the behavior.
\begin{lemma}[Behavioral lift]
\label{lem:lift}
    Let $w\in \mathcal C^\infty(\overline{\I},\R^{m+p})$. Then $w\in\B$ if and only if
    \begin{equation}
        R(\mathcal D) \Pi w = 0.
    \end{equation}
\end{lemma}

\begin{example}[No finite expansion]\label{ex:nofin}
    We show for the linear time-invariant system described by \eqref{eq:sys} with $A=C=I$ and $B=D=0$ that its trajectories with nontrivial output have  no series expansion involving only finitely many polynomials $\pi_i$. Let $\col(u,y)\in\B$ and assume that $y$ has a finite expansion, i.e.\ there is some $N\in\N$ such that $\hat y_i=0$ for $i\geq N$. It is no restriction to assume $\col(u,y)\in\B\cap \mathcal C^\infty(\overline{\I},\R^q)$. Note that the kernel representation~\eqref{eq:ker} of $\B$ is given via $R(s)=\begin{bmatrix}
        0 & (s-1)I
    \end{bmatrix}$. With \Cref{lem:lift} and the definition of $\mathcal D$ in~\eqref{eq:infMat} we see that
    \begin{equation}
    \label{eq:summ}
    \hat y_i = (\mathcal D\hat y)_i = (2i+1)\sum_{\substack{j=i+1\\i+j\text{ odd}}}^\infty \hat y_j,\quad i\in\N.
    \end{equation}
    The finiteness of the expansion yields that for $i=N-1$ all summands on the right hand side in~\eqref{eq:summ} vanish and, hence, $\hat y_{N-1}=0$. It is not difficult to see that this successively implies $\hat y_i = 0$ for all $i\in\N$. Therefore $y$ is trivial.\hfill \qed
\end{example}

In \Cref{ex:nofin} we illustrated the case when the dynamics of  a particular linear time-invariant system \eqref{eq:sys} cannot be described by a finite representation in terms of Legendre polynomials. This is a generic situation: the Legendre series representation of  exponential functions $e^{\lambda t}$, $\lambda\neq 0$ (generically present  in the solution of \eqref{eq:sys}, e.g. in the free response), involve an infinite number of terms (see e.g. p. 39 of \cite{Lyust63}). Such considerations lead naturally to working with truncated Legendre expansions of solutions of \eqref{eq:sys}. 

\subsection{Truncated expansion}
\label{subsec:trunc}
In the light of \Cref{ex:nofin}, we study approximation bounds when considering truncated series of Legendre polynomials. To this end we introduce the orthogonal projection $P_N: L^2(\I, \R^d)\rightarrow L^2(\I, \R^d)$ defined by 
\begin{equation}
    P_N f := \sum_{i<N} \hat f_i \pi_i,
\end{equation}
Note that $\im P_N$ coincides with the $N$-dimensional space of $\R^{d}[s]$-polynomials of degree up to $N-1$, which is spanned by $\pi_0,\dots, \pi_{N-1}$.
Since $(\pi_i)_{i\in\N}$ is an orthogonal basis, $P_N f$ converges to $f$ as $N\rightarrow\infty$ with respect to the $L^2$-norm.  The speed of convergence is related to the smoothness of~$f$, as we discuss in the following.

    Recall that $\pi_i$ is an eigenfunction corresponding to the $i$th eigenvalue $\lambda_i := i(i+1)$ of the Sturm--Liouville operator
\begin{equation*}
    \begin{split}
    \mathcal Lf &= \ell f:= -\frac{\mathrm d}{\mathrm dt}\left( p\frac{\mathrm d}{\mathrm dt} f\right),\quad p(t) = (1-t^2),\\
    D(\mathcal L) &:= \left\{f\in L^2(\I,\C)\,\middle|\, \begin{aligned}f,pf^{(1)}\in\mathcal {AC}(\I,\R),\\ \ell\in L^2(\I,\C),\\ pf^{(1)}(-1) = pf^{(1)}(1) = 0\end{aligned}\right\},
\end{split}
\end{equation*}
see \cite[Theorem~3.6]{littlejohn2011legendre}. Here, $\mathcal{AC}(\I,\C)$ denotes the space of locally absolutely continuous functions from $\I$ to $\C$. Let $\mathcal L^s$ for $s\in (0,\infty)$ denote the $s$th power the self-adjoint operator $\mathcal L$, defined via functional calculus, see e.g.\ Section~5.3 in \cite{schmudgen2012unbounded}.
\begin{lemma}
    If $f\in D(\mathcal L^s)$ for some $s>0$, then
    \begin{equation}
        \norm*{(I-P_N)f} = \norm{\mathcal L^{s}f}\mathcal O(N^{-2s})\qquad (N\rightarrow \infty).
    \end{equation}
    Moreover, $H^k(\I,\C)\subset D(\mathcal L^{\frac{k}{2}})$ for $k\in\N$.
\end{lemma}
    \begin{proof}
        Let $\tilde \pi_i = \pi_i/\norm{\pi_i}$, i.e.\ $(\tilde \pi_i)_{i\in\N}$ form an orthonormal basis in $L^2(\I,\C)$. For $f\in D(\mathcal L^s)$ we have
        \begin{align*}
            &\norm{(I-P_N)f}^2 =
            \norm[\Bigg]{f-\sum_{i<N} \scp{f,\tilde\pi_i}\tilde\pi_i}^2\\
            &= \norm[\Bigg]{\sum_{i\geq N} \scp{f,\tilde \pi_i}\tilde \pi_i}^2
            =\norm[\Bigg]{\sum_{i\geq N} \lambda_i^{-s}\scp{f,\mathcal L^s\tilde\pi_i}\tilde\pi_i}^2 \\
            &=\norm[\Bigg]{\sum_{i\geq N} \lambda_i^{-s}\scp{\mathcal L^sf,\tilde\pi_i}\tilde\pi_i}^2.
        \end{align*}
        Since $(\lambda_i)_{i\in\N}$ is an increasing sequence, we find
        \begin{align*}
            \norm{(I-P_N)f}^2 &\leq \lambda_{N}^{-2s} \norm[\Bigg]{\sum_{i\geq N} \scp{\mathcal L^sf,\tilde\pi_i}\tilde\pi_i}^2 \\
            &\leq
            \big(N(N+1)\big)^{-2s} \norm{\mathcal L^s}^2,
        \end{align*}
        which shows the first claim. 
        
        We show the inclusion $H^k(\I,\C)\subset D(\mathcal L^{\frac{k}{2}})$, $k\in\N$. For $k=0$ there is nothing to prove. For $k=1$ we find that $D(\mathcal L^{\frac{1}{2}})=\{f\in L^2(\I,\C)\cap \mathcal {AC}(\I,\C)\,|\, \sqrt{p}f^{(1)}\in L^2(\I,\C)\}$ by \cite[Theorem~6.8.5~(i)]{behrndt2020boundary}. With the uniform boundedness of $p$ on $[-1,1]$ this shows $H^1(\I,\C)\subset D(\mathcal L^{\frac{1}{2}})$.
        
        We continue with the case of even $k\geq 2$. Let $f\in H^k(\I,\C)$. Then it is clear that $f$,~$pf^{(1)}\in\mathcal{AC}(\I,\C)$ and $\ell f \in L^2(\I,\C)$. Moreover, $f^{(1)}$ is bounded on $[-1,1]$ as $f^{(1)}\in H^1(\I,\C)$. Consequently, $f\in D(\mathcal L)$ and $\mathcal L f \in H^{k-2}(\I,\C)$. Repeating this argument yields $f\in D(\mathcal L^{\frac{k}{2}})$, showing $H^k(\I,\C)\subset D(\mathcal L^{\frac{k}{2}})$ for even $k$. 
        
        Finally, we consider $f\in H^{k+1}(\I,\C)$ for even $k\geq 2$. From the previous observations we know that $f\in D(\mathcal L^{\frac{k}{2}})$ and $\mathcal L^{\frac{k}{2}}f\in H^1(\I,\C)\subset D(\mathcal L^\frac{1}{2})$. Therefore, $f\in D(\mathcal L^\frac{k+1}{2})$, which shows $H^{k+1}(\I,\C)\subset D(\mathcal L^\frac{k+1}{2})$.
    \end{proof}

\begin{corollary}
    \label{lem:convRate}
    If $f\in H^k(\I,\R^d)$ for some $k\in\N$, then
    \begin{equation}
        \norm{(I-P_N)f} = \mathcal O(N^{-k})\qquad (N\rightarrow \infty).
    \end{equation}
\end{corollary}

\subsection{Polynomial trajectories}\label{subsec:polytraj}
Next we show that the space of polynomial trajectories $\B\cap \bigcup_{N\in\N}\im P_N$ is dense in $\B\cap H^s(\I,\R^q)$, $s\in\N$, and, particularly, in $\B$.
\begin{proposition}
    \label{lem:poly_traj1}
        Suppose $\B$ is controllable and let $w\in\B\cap H^{\n(\B)+s+k}(\I,\R^{m+p})$ for some $s\in\N\setminus\{0\}$, $k\in\N$. For all $N\in\N$, $N\geq \n(\B)+s+1$, there is $w^N\in\B\cap \im P_N$ such that \begin{equation*}
            \Lambda_{s}(w^N-w)(-1)=0
        \end{equation*}
        satisfying
    \begin{equation}
    \label{eq:w-wN}
        \norm{w-w^N}_{H^{s-1}}= \mathcal O(N^{-k})\qquad (N\rightarrow\infty).\vspace*{1mm}
    \end{equation}
\end{proposition}
\noindent The integers $k\in\N$ and $s\in\N\setminus\{0\}$ in \Cref{lem:poly_traj1} determine the convergence order and the highest derivative up to which the asymptotic behavior is valid, cf.\ \eqref{eq:w-wN}. These can be considered as user specifiable, provided  $w$ is sufficiently smooth.
\begin{proof}[Proof of \Cref{lem:poly_traj1}]
    Controllability of $\B$ implies the existence of an image representation~\eqref{eq:image}. Here we consider a particular image representation given by a flat output, see \Cref{lem:flatimage}. Let $w\in \B\cap H^{\n(\B)+s+k}(\I, \R^{m+p})$. Then there is $\ell\in H^{\n(\B)+s+k+1}(\I,\R^{m})$ such that
    \begin{equation}
        \label{eq:imag_proof}
        w = M\left(\frac{\mathrm d}{\mathrm dt}\right) \ell
    \end{equation}
    holds. We construct a polynomial which approximates $\ell$ and its derivatives up to order $\gamma:=\n(\B)+s+1$, while matching the initial values. Let $v_{\gamma}^N := P_{N-\gamma} \ell^{(\gamma)}$. Then by \Cref{lem:convRate} as $N\rightarrow \infty$ one has
    \begin{equation}
        \norm{v_{\gamma}^N - \ell^{(\gamma)}} = \norm{(I-P_{N-\gamma})\ell^{(\gamma)}} = \mathcal O(N^{-k}).
    \end{equation}
    Define
    \begin{equation}
        v_i^N(t) := \ell^{(i)}(-1) + \int_{-1}^t v_{i+1}^N(\tau)\,\mathrm d\tau,\quad i=0,\dots,\gamma-1.
    \end{equation}
    By construction $(v_0^N)^{(i)} = v_i^N$, $(v_0^N)^{(i)}(-1) = v_i^N(-1) = \ell^{(i)}(-1)$, i.e.\ 
    \begin{equation}
        \label{eq:helpfull}
        \Lambda_{\gamma}(v_0^N-\ell)(-1) = 0.
    \end{equation}
    Moreover, one sees with the Cauchy-Schwarz inequality
    \begin{equation*}
    \begin{split}
        \norm{\ell^{(i)}-(v_0^N)^{(i)}}^2 &= \int_{\I} \abs*{\int_{-1}^\tau \ell^{(i+1)}(t)-(v_0^N)^{(i+1)}(t)\,\mathrm dt}^2\,\mathrm d\tau\\
        &\leq 4 \norm{\ell^{(i+1)}-(v_0^N)^{(i+1)}}^2
    \end{split}
    \end{equation*}
    and, thus,
    \begin{equation}
        \label{eq:helpfull_season_2}
        \norm{\ell-v_0^N}_{H^\gamma}  = \mathcal O(N^{-k}).
    \end{equation}
    By construction $v_0^N\in\im P_N$.
    Recall that the polynomial matrix $M$ in \eqref{eq:imag_proof} satisfies $\deg(M)\leq \n(\B)+1$. Let $w^N=M(\frac{\mathrm d}{\mathrm d t}) v_0^N$, which is an element of $\B\cap \im P_N$. Now, \eqref{eq:helpfull} and \eqref{eq:helpfull_season_2} yield
    \begin{equation*}
        \Lambda_{s}(w^N-w)(-1) = \Lambda_{s}\bigl(M\tderb(v_0^N-\ell)\bigr)(-1)=0
    \end{equation*} and
    \begin{equation*}
        \norm{w^N-w}_{H^{s-1}} = \norm{M(\tfrac{\mathrm d}{\mathrm dt})(v_0^N-\ell)}_{H^{s-1}} = \mathcal O(N^{-k}).\qedhere
    \end{equation*}
\end{proof}

\section{The LQR problem and its approximation}  \label{sec:LQR_approx}

Our aim is to solve the quadratic optimal control problem
\begin{subequations}
    \label{eq:LQR}
\begin{align}
    \label{eq:LQR1}
        \operatorname*{minimize}_{w\in\B\cap H^{\lag(\B)}(\I,\R^{q})}\ J(w)\quad\text{s.t.\ }\\
        \label{eq:LQR2}
        \Lambda_{\lag(\B)}(w)(-1) = \xi^0,
\end{align}
\end{subequations}
with the cost function 
\begin{equation}
    \label{eq:cost}
    J(w):=\norm{y}^2 + \norm{u^{(\lag(\B))}}^2,\quad w=\col(u,y)\in\B.
\end{equation}
The initial condition~\eqref{eq:LQR2} uniquely determines the latent state, provided the latter is observable from the inputs and the outputs. 
Including the higher-order derivative term of the input into the objective function~\eqref{eq:cost} ensures feasibility of the LQR problem.
 
\begin{lemma}
    \label{lem:LQRsmoothness}
    Problem~\eqref{eq:LQR} has a unique solution $w^\star$, and $w^\star\in\mathcal C^\infty(\overline{\I},\R^q)$. Moreover, every feasible trajectory $w$ satisfies
    \begin{equation}
        \label{eq:strongconv}
        J(w-w^\star) \leq 2\bigl(J(w)-J(w^\star)\bigr).
    \end{equation}
\end{lemma}
\begin{proof}
    We fix a minimal input-state-output representation~\eqref{eq:sys}, that is $(A,C)$ is observable. Consider any $w=\col(u,y)\in\B\cap H^{\lag(\B)}$. Then there is $x\in H^1(\I,\R^n)$ satisfying~\eqref{eq:sys}. With
    \begin{equation}
        \label{eq:T}
        \mathcal T_k := \begin{cases}
                D&\text{if }k=0,\\
                \begin{bmatrix}
                    D\\
                    \mathcal O_{k-1}B & \mathcal T_{k-1}
                \end{bmatrix}&\text{if }k\geq 1,
            \end{cases}
    \end{equation}
    with $\mathcal O_k$ being the Kalman observability matrix defined in \eqref{eq:observstate}
    one has
    \begin{equation}
        \label{eq:lamby}
        \Lambda_{k+1}(y) = \mathcal O_{k} x + \mathcal T_{k} \Lambda_{k+1} (u)
    \end{equation}
    and by employing observability
    \begin{equation}
        \label{eq:observstate}
        x = \mathcal O^{\dagger}_{\lag(\B)-1} \bigl(\Lambda_{\lag(\B)}(y) - \mathcal T_{\lag(\B)-1}\Lambda_{\lag(\B)}(u)\bigr).
    \end{equation}
    Inserting \eqref{eq:observstate} into \eqref{eq:lamby} for $k=\lag(\B)$ by rearranging terms one obtains a linear auxiliary system \begin{equation}
        \label{eq:aux}
        \frac{\diff}{\diff t} \xi = \tilde A \xi + \tilde B \nu
    \end{equation} with state $\xi = 
        \Lambda_{\lag(\B)}(w)$, input $\nu=u^{(\lag(\B))}$. Note that~\eqref{eq:LQR} is equivalent to the LQR problem
    \begin{equation*}
        \operatorname*{minimize}_{\xi,\nu}\,\int_\I \xi(t)^\top \mathcal Q \xi(t) + \nu(t)^\top \nu(t)\,\mathrm dt,
    \end{equation*}
    where $\mathcal Q = \operatorname{diag}(0, I_p, 0,\dots, 0)$, subject to the dynamics~\eqref{eq:aux} and the initial condition $\xi(-1) = \xi^0$. By standard LQR theory the latter problem has a solution $(\xi^\star,\nu^\star)$, which is infinitely differentiable as $\nu^\star$ is a state feedback involving a solution of a Riccati differential equation. In particular, $w^\star :=\operatorname{diag}(I_m,I_p,0,\dots, 0)\xi^\star\in\mathcal C^\infty(\overline{\I},\R^{q})$ solves \eqref{eq:LQR}.

    Let $w$ be any trajectory. It is not difficult to see that $J$ satisfies the parallelogram identity
    \begin{equation*}
        J\bigl(\tfrac{1}{2}(w^\star + w)\bigr) + J\bigl(\tfrac{1}{2}(w^\star - w)\bigr)= \tfrac{1}{2} J(w^\star) + \tfrac{1}{2} J(w).
    \end{equation*}
    Convexity of the feasibility region together with $J(w^\star)\leq  J(\tfrac{1}{2}(w^\star + w))$ yield
    \begin{equation*}
        J\bigl(\tfrac{1}{2}(w^\star - w)\bigr) \leq \tfrac{1}{2} \bigl(J(w)-J(w^\star)\bigr).
    \end{equation*}
    This shows \eqref{eq:strongconv}.
\end{proof}

Instead of solving the OCP~\eqref{eq:LQR} directly, given $N \in \N$, we solve the problem restricted to polynomial trajectories, i.e.,
\begin{subequations}
    \label{eq:LQRapprox}
\begin{align}
    \label{eq:LQRapprox1}
        \operatorname*{minimize}_{w\in\B\cap \im P_N}\ J(w)\quad\text{s.t.\ }\\
        \label{eq:LQRapprox2}
        \Lambda_{\lag(\B)}(w)(-1) = \xi^0.
\end{align}
\end{subequations}
Observe that the restriction $w\in\B\cap \im P_N$ enforces polynomial trajectories of degree at most $N-1$. In the following we show that solving \eqref{eq:LQRapprox} leads to an approximately optimal control and, as $N\rightarrow\infty$, the optimality gap decays at an polynomial rate of arbitrary order.

\begin{theorem}[Convergence of optima]\label{thm:conv_optm}
    For given approximation order~$N$, $N \in \N$, let $w^\star$ and $w^N$ be the solutions to the OCPs~\eqref{eq:LQR} and~\eqref{eq:LQRapprox}, respectively. For every $k\in\N$ one has
    \begin{equation}
        \label{eq:asymp}
        0\leq J(w^N) - J(w^\star)  = \mathcal O(N^{-k})
    \end{equation}
    and
    \begin{equation}
        \label{eq:asymp3}
        \norm{w^\star-w^N} = \mathcal O(N^{-k})
    \end{equation}
    as $N\rightarrow\infty$.
\end{theorem}

\begin{proof}
    Recall that the solution $w^\star\in\mathcal{C}^\infty(\overline{\I},\R^{m+p})$, see \Cref{lem:LQRsmoothness}. Thus, by \Cref{lem:poly_traj1} for arbitrary $k\in\N$ there is $v^N\in\B\cap \im P_N$ with $\Lambda_{\mathfrak{l}(\B)}(w^\star-v^N)(-1)=0$ and
    \begin{equation}
    \label{eq:asymp2}
        \norm{w^\star-v^N}_{H^{\lag(B)}}=\mathcal O(N^{-k}).
    \end{equation}
    As $w^\star$ and $w^N$ are solutions of \eqref{eq:LQR} and \eqref{eq:LQRapprox}, respectively, one has
    \begin{equation*}
        J(w^\star) \leq J(w^N) \leq J(v^N), \quad N\in\N,
    \end{equation*}
    which shows the left-hand-side inequality in~\eqref{eq:asymp}. Further, by the reverse triangle inequality
    \begin{equation*}
        \abs*{J(v^N)^\frac12-J(w^\star)^\frac12}\leq J(v^N-w^\star)^{\frac12} \leq \norm{v^N-w^\star}_{H^{\lag(\B)}},
    \end{equation*} which together with \eqref{eq:asymp2} implies
    \begin{align*}
        J(v^N) - J(w^\star) ={ }& 2J(w^\star)^{\frac12}\bigl(J(v^N)^{\frac12}-J(w^\star)^{\frac12}\bigr)\\
        &{ } + \bigl(J(v^N)^{\frac12}-J(w^\star)^{\frac12}\bigr)^2\\
        ={ }& \mathcal O(N^{-k}).
    \end{align*}
    Next we show \eqref{eq:asymp3}. Let $\col(u^\star,y^\star)=w^\star$ and $\col(u^N,y^N)=w^N$. By $\Lambda_{\lag(\B)}(w^\star-w^N)(-1)=0$. We find
    \begin{align*}
        &\norm{(u^\star)^{(j)}-(u^N)^{(j)}}^2 \\&= \int_\I \abs*{\int_{-1}^\tau (u^\star)^{(j+1)}(t)-(u^N)^{(+1)}(t)\,\diff t}^2\,\diff \tau\\
        &\leq 4 \norm{(u^\star)^{(j+1)}-(u^N)^{(j+1)}}^2
    \end{align*}
    for all $j=0,\dots,\lag(\B)-1$. Thus,
    \begin{equation*}
        \norm{w^\star-w^N}^2 \leq 4^{\lag(\B)} J(w^\star-w^N). 
    \end{equation*}
    This together with \eqref{eq:strongconv} in \Cref{lem:LQRsmoothness} and \eqref{eq:asymp} yields \eqref{eq:asymp2}. 
\end{proof}

\begin{remark}
    \label{rem:cost}
    Similar to the above approach, one can handle a cost function given by any \emph{quadratic differential form}, see \cite{QDF},
    \begin{equation}
        J(w) = \sum_{i,j=0}^{\lag(\B)} (w^{(i)})^\top \Phi_{i,j} w^{(j)},
    \end{equation} with matrices $\Phi_{i,j}\in\R^{q\times q}$ such that $\Phi_{i,i}=\Phi_{i,i}^\top$ and
    \begin{equation}
        \Phi_{\lag(\B),\lag(\B)} = \begin{bmatrix}
            \widetilde \Phi & 0\\
            0 & 0
        \end{bmatrix},
    \end{equation}
    where $\widetilde\Phi\in\R^{m\times m}$ corresponding to $u^{(\lag(\B))}$ is invertible.
\end{remark}

\begin{remark}
    \label{rem:cost_nofeed}
    In the case where the state is directly observable at the output, i.e.\ $C=I_n$ and $D=0$ in the input-state-output representation~\eqref{eq:sys}, \Cref{lem:LQRsmoothness} and \Cref{thm:conv_optm} likewise apply to the LQR problem
    \begin{subequations}
        \label{eq:LQR_nofeed}
    \begin{align}
        &\operatorname*{minimize}_{\col(u,x)\in\B}\ \norm{x}^2 + \norm{u}^2\quad\text{s.t.}\\
        &x(-1) = x^0
    \end{align}
    \end{subequations}
    and its restrictions to polynomial trajectories.
\end{remark}

\section{A ``fundamental lemma"}  \label{sec:f-lemma}

The main result of this section is a parametrization of the trajectories of a controllable linear differential system in terms of a constant matrix obtained from ``sufficiently-informative" data. 
To this end, we first define some new concepts and notation and state some preliminary results. 

\subsection{Persistency of excitation}\label{sec:PE}

Given $L \in \N \setminus \{0\}$ and $f\in H^{L-1}(\I,\R^d)$, we define the Gramian 
\begin{equation}\label{eq:Gammaf}
      \Gamma_{L}(f) := \int_{\I} \Lambda_{L} (f)\Lambda_{L} (f)^\top\,\mathrm d t.
\end{equation}
\begin{definition}
    Let $L\in\N\setminus\{0\}$. A function $f:\I\rightarrow\R^d$ is called \emph{persistently exciting of order $L$}, if $f\in H^{L-1}(\I,\R^d)$ and the Gramian $\Gamma_{L}(f)$ in \eqref{eq:Gammaf} is positive definite.
\end{definition}
This definition is reminiscent of the notion of \emph{excitation} in \cite[Definition~2]{mareels1984sufficiency}. In the following result we relate it to the concept of persistency of excitation used in \cite{rapisarda2022persistency}, specifically property (3) in the lemma below.
\begin{lemma}\label{lem:eqDef}
    For $f \in H^{L-1}(\I,\R^d)$ with $L \in \N \setminus\{0\}$, the following statements are equivalent:
    \begin{itemize}
        \item[(1)] $f$ is persistently exciting of order $L$;
        \item[(2)] $\ker(\Gamma_k(L))=\{0\}$;
        \item[(3)] If $\eta\in \R^{Ld}$ is such that $\eta^\top \Lambda_L(f)=0$ a.e., then $\eta=0$;
        \item[(4)] The functions $f,f^{(1)},\dots, f^{(L-1)}$ are linearly independent in $L^2(\I,\R^d)$. 
    \end{itemize}
    \end{lemma}
\begin{proof}
    We show only the equivalence of (2) and (3), as the equivalence of (1) and (2) as well as (3) and (4) are straightforward. Observe that $\eta\in\ker(\Gamma_L(f))$ if and only if
    \begin{equation*}
        0 = \eta^\top \Gamma_L(f)\eta = \int_\I \norm{\Lambda_L(f)(t)^\top\eta}^2_2\,\mathrm dt = \norm{\Lambda_L(f)^\top\eta}^2,
    \end{equation*}
    which shows the equivalence of (2) and (3).
\end{proof}

\begin{example}
    \label{ex:pe_poly}
    Let $f:\I\rightarrow \R$ be a monic polynomial of degree $L-1$ with $L\in\N\setminus\{0\}$. Then $f^{(k)}$ for $k=0,\dots,L-1$ is a polynomial of degree $L-1-k$. Therefore, it is clear that $f,f^{(1)},\dots, f^{(L-1)}$ are linearly independent functions in $L^2(\I,\R)$, and, thus, $f$ is persistently exciting of order $L$ by \Cref{lem:eqDef}.
\end{example}

\begin{remark}[Discrete-time excitation analogue]
    The above concept of persistency of excitation in continuous time aligns seamlessly with its discrete-time counterpart (see \cite{willems2005note}). In discrete time, the time-shift operator serves as the analogue to differentiation, whereas summation corresponds to integration. 
    With this in mind, given a discrete-time signal 
    $f:\{0,\dots,N-1\}\rightarrow \R^d$, we define 
    \[
        \Lambda_L(f)(t) := \begin{bmatrix}
            f^\top(t) & \dots & f^\top(t+L-1)
        \end{bmatrix}{}^\top
    \]
    and consider the Hankel matrix
    \begin{equation*}
        H_L(f) := \begin{bmatrix}
            \Lambda_L(f)(0)  &\dots & \Lambda_L(f)(N-L)
        \end{bmatrix}.
    \end{equation*}
Now, the corresponding Gramian
\begin{align*}
    \Gamma_L(f)&:= \sum_{j=0}^{N-L} \Lambda_L(u)(j) \Lambda_L(u)(j)^\top=H_L(u)H_L(u)^\top
\end{align*} is positive definite if and only if $H_L(u)$ has full row rank, i.e.\ $f$ is persistently exciting of order $L$, cf.~\cite{willems2005note}.
\end{remark}

\subsection{Fundamental lemma}

In the following, we also need Gramians constructed from input-state trajectories of a system \eqref{eq:sys}. 
Given $u\in H^{L-1}(\I,\R^m)$ and $x\in H^{K-1}(\I,\R^n)$ for some $L,K\in\N\setminus\{0\}$, we extend the notation of stacked derivatives~\eqref{eq:Lambda} to 
\begin{equation}
    \label{eq:Lambdaxu}
    \Lambda_{L,K}(u,x):=\begin{bmatrix}
        \Lambda_{L} (u) \\ \Lambda_{K} (x)
    \end{bmatrix}
\end{equation}
We define the Gramian $\Gamma_{L,K} (x,u)$
by 
\begin{equation}
    \label{eq:Gammaux}
    \Gamma_{L,K} (u,x): = \int_\I \Lambda_{L,K} (u,x)\Lambda_{L,K} (u,x)^\top\,\diff t\; .
\end{equation}

We now state some results instrumental to establishing
a continuous-time fundamental lemma. 
To this end, we consider an input-state-output representation~\eqref{eq:sys}. Then, for fixed $L \in \N \setminus\{0\}$, we define for $i \in \N$
\begin{equation}
\mathcal C_i(A,B) := \begin{cases}
        I_{n+Lm}&\text{if } i=0\\
    \begin{bmatrix}
        A^i & A^{i-1} B &\dots & B & 0\\
        0& 0 & \dots & 0 & I_{Lm}
    \end{bmatrix} &\text{if }i\geq 1
\end{cases}\; .
\end{equation}

\begin{lemma}\label{prop:diffGamma}
    Consider an input-state-output representation~\eqref{eq:sys} of $\B$ and let $\col(u,x)$ be an input-state trajectory with $u\in H^{L+n}(\I,\R^{m})$. For $i=0,\ldots,n-1$, the following equalities hold:
\begin{equation}\label{eq:Lambdasfurther}
\begin{bmatrix}
\Lambda_1(x^{(i)})\\
\Lambda_{L}(u^{(i)})
\end{bmatrix}=\mathcal{C}_i(A,B) \begin{bmatrix}
    \Lambda_1(x)\\
    \Lambda_{L+i}(u)
    \end{bmatrix}\; .
\end{equation}
\end{lemma}

\begin{proof}
    The case $i=0$ is trivial. To prove \eqref{eq:Lambdasfurther} in the  case $i\geq 1$, use the equation 
\begin{equation*}
x^{(i)}=A^i x+\sum_{j=0}^{i-1} A^{i-1-j}B u^{(j)}.\qedhere
\end{equation*}
\end{proof}

The next result is analogous to \cite[Proposition~1]{rapisarda2022persistency}; since the proof needs to be adapted to the language and notation of this paper, we provide it in full detail.
\begin{proposition}\label{prop:PE1}
Suppose that $\B$ is controllable and consider a minimal input-state-output representation~\eqref{eq:sys} of $\B$ such that $(A,B)$ is controllable. Let $\col(x,u)$ be a input-state trajectory. Assume that $u$ is persistently exciting of order at least $n+L$, with $L\in\N\setminus\{0\}$. Then 
\begin{itemize}
    \item[(1)] If $\xi\in\R^{Lm+n}$ satisfies
\begin{equation}
\label{eq:ortho}
\xi^\top\Lambda_{L,1}(u,x)=0
\end{equation}
almost everywhere on $\I$, then $\xi=0$; 
\item[(2)] $\Gamma_{L,1}(u,x)$ is positive definite.
\end{itemize}
\end{proposition}

\begin{proof}
The second statement follows in a straightforward way from the first one, cf.\ proof of \Cref{lem:eqDef}. We show (1).
Let $\xi^\top = \begin{bmatrix} \eta & \zeta\end{bmatrix}$, $\eta = \begin{bmatrix} \eta_0 & \dots & \eta_{L-1}\end{bmatrix}$, with $\eta_k\in\R^{1\times m}$, $j=0,\dots,L-1$, and $\zeta\in\R^{1\times n}$. Differentiating \eqref{eq:ortho} $i$ times, $i=0,\ldots,n$, we conclude that 
\[
\begin{bmatrix}\zeta & \eta\end{bmatrix} \begin{bmatrix}
    \Lambda_1(x^{(i)})\\
    \Lambda_{L}(u^{(i)})
    \end{bmatrix}=0\; , 
\]
almost everywhere on $\I$ for $i=0,\ldots,n$.

Using equation \eqref{eq:Lambdasfurther} established in the proof of \Cref{prop:diffGamma}, we conclude that for $i=0,\ldots,n$ it  that 
\begin{equation}\label{eq:istep}
    0=\begin{bmatrix} \zeta  A^i&\ldots& \zeta  B& \eta_0  &\ldots& \eta_{L-1}   \end{bmatrix} \begin{bmatrix} \Lambda_1(x)\\  \Lambda_{L+i}(u)\end{bmatrix}\; ,
\end{equation}
holds almost everywhere on $\I$. Now define 
\begin{eqnarray*}
w_0&:=&\begin{bmatrix}  \zeta  & \eta_0  &\ldots& \eta_{L-1}  & 0_{nm}\end{bmatrix}\\
w_1&:=&\begin{bmatrix}  \zeta  A & \zeta  B& \eta_0  &\ldots& \eta_{L-1}   &0_{(n-1)m} \end{bmatrix}\\
&\vdots&\\
w_n&:=&\begin{bmatrix}  \zeta  A^{n}  \ldots& \zeta  B& \eta_0  &\ldots& \eta_{L-1}  \end{bmatrix}\; .
\end{eqnarray*}
From  \eqref{eq:istep} we have that  the following equations hold true almost everywhere on $\I$: 
\begin{equation}\label{eq:otherHankel}
w_i\begin{bmatrix}
\Lambda_1(x)\\ \Lambda_{L+n}(u)
\end{bmatrix}=0\; , \; i=0,\ldots,n\; .
\end{equation}
Since $u$ is persistently exciting of order at least $L+n$, using statement 3 of \Cref{lem:eqDef} we conclude that  the vector-valued function $\begin{bmatrix}
\Lambda_1(x)^\top & \Lambda_{L+n}(u)^\top
\end{bmatrix}{}^\top$ has at most $n$ ``almost everywhere annihilators"  on $\I$: it follows that the $n+1$ vectors $w_i$, $i=0,\ldots,n$ are linearly dependent. 

Since the last components of the $w_i$'s are zero, $i=0,\ldots,n$,  we conclude that $ \eta_{L-1}  =0_{1\times m}$, then $\eta_{L-2}=0$, and so on until $ \eta_0  =0$. Consequently 
\begin{eqnarray*}
w_0&=&\begin{bmatrix}  \zeta  &0_{1\times (n+L)m}\end{bmatrix}\\
w_1&=&\begin{bmatrix}  \zeta  A & \zeta  B&0_{1\times (n+L-1)m}\end{bmatrix}\\
w_2&=&\begin{bmatrix}   \zeta  A^2 & \zeta  AB&  \zeta  B&0_{(n+L-2)m}  \end{bmatrix}\\
&\vdots&\\
w_n&=&\begin{bmatrix}  \zeta  A^{n} & \zeta  A^{n-1} B& \ldots& \zeta  B&0_{1\times Lm}\end{bmatrix}\; .
\end{eqnarray*}
Denote by $\alpha_i$, $i=0\ldots,n$ the coefficients of the characteristic polynomial of $A$, and using $\sum_{i=0}^n  A^i \alpha_i=0$ conclude that $\sum_{i=0}^n w_i \alpha_i$ equals 
\begin{eqnarray*}
\begin{bmatrix}\sum_{i=0}^n  \zeta A^i \alpha_i&\sum_{i=1}^n \alpha_i \zeta A^{i-1}B&\ldots &\zeta B&0_{1\times Lm}\end{bmatrix}\\
= \begin{bmatrix}0_{1\times n}&\sum_{i=1}^n \zeta   \alpha_i A^{i-1}B&\ldots &\zeta  B&0_{1\times Lm}\end{bmatrix}\; . 
\end{eqnarray*}
By construction, almost everywhere on $\I$ it holds that 
\[
\begin{bmatrix}\sum_{i=1}^n \alpha_i  \zeta   A^{i-1}B&\ldots &\alpha_n  \zeta   B\end{bmatrix}
\Lambda_{n}(u)=0\; ;
\] 
since $u$ is persistently exciting of order at least $L+n$, we conclude that 
\[
 \begin{bmatrix}\sum_{i=1}^n \alpha_i \zeta A^{i-1}B&\sum_{i=2}^n \alpha_i \zeta A^{i-2}B &\ldots&\alpha_n \zeta B\end{bmatrix}=0\; .
\]
It follows from the last $m$ equations that $\alpha_n   \zeta  B=0$; since the highest coefficient $\alpha_n$ of the characteristic polynomial of $A$ equals 1, we conclude that $ \zeta  B=0$. The previous $m$-dimensional block-entry of the vector is $\alpha_{n-1}  \zeta  B+\alpha_n \zeta  AB=0+\alpha_n\zeta AB=0$. We conclude that  $ \zeta  AB=0$. The same argument can be used to prove $ \zeta  A^i B=0$, $i=0,\ldots,n-1$. Since the pair $(A,B)$ is controllable we conclude that $\zeta=0$ and consequently that Statement (1) is true. 
\end{proof}

We now have all the necessary ingredients to formulate a continuous-time ''fundamental lemma``.
\begin{theorem}[Continuous-time ''fundamental lemma``]
    \label{thm:FL}
    Suppose that $\B$ is controllable. Let $\col(\overline u,\overline y)\in \B$ be such that $\overline u$ is persistently exciting of order $L+\n(B)$, with $L\geq \lag(\B)+1$.
    For $\col(u,y)\in H^{L-1}(\I,\R^q)$ and $K\in\N$, $\lag(\B)+1\leq K\leq L$, the following statements are equivalent:
    \begin{itemize}
        \item[(1)] $\col(u,y)\in\B$;
        \item[(2)] There exists $g\in L^2(\I, \R^{Lm+Kp})$ such that
        \begin{equation}
            \label{eq:FLrep}
                \Lambda_{L,K}(u,y) = \Gamma_{L,K}(\overline{u},\overline{y}) g.
        \end{equation}
    \end{itemize}
    Moreover, $\rank \Gamma_{L,K}(\overline{u},\overline{y}) = Lm + \n(\B)$.
\end{theorem}
\begin{proof}
    Fix a minimal input-state-representation~\eqref{eq:sys} of $\B$ and let
    \begin{equation}
        \mathcal S_{L,K} = \begin{bmatrix}
            I_{mL} & 0\\
            \mathcal T_{K-1} & \mathcal O_{K-1}
        \end{bmatrix},
    \end{equation}
    where $\mathcal O_K$ is the Kalman observability matrix, see \eqref{eq:Ok}, and $\mathcal T_K$ is defined as in \eqref{eq:T}. Then given $\col(u,y)\in\B$ with corresponding state $x$ satisfies
    \begin{equation}
        \label{eq:SLambda}
        \Lambda_{L,K}(u,y) = \mathcal S_{L,K} \Lambda_{L,1}(u,x).
    \end{equation}
    
    Let $\overline x$ be the state corresponding to $\col(\overline u,\overline y)$. In a first step we show
    \begin{equation}
        \label{eq:images}
        \im \mathcal S_{L,K} = \im \Gamma_{L,K}(\overline u,\overline y).
    \end{equation} 
    Note that in order to show \eqref{eq:images} it suffices to prove
    \begin{equation}
        \label{eq:kernels}
        \ker \mathcal S_{L,K}^\top = \ker \Gamma_{L,K}(\overline u,\overline y).
    \end{equation}
    The former equality then follows by taking the orthogonal complements of the null spaces and employing the symmetry of $\Gamma_{L,K}(\overline u,\overline y)$. With \eqref{eq:SLambda} and \eqref{eq:Gammaux} one has
    \begin{equation}
        \Gamma_{L,K}(\overline u,\overline y) = \mathcal S_{L,K} \Gamma_{L,1}(\overline u,\overline x)\mathcal S_{L,K}^\top.
    \end{equation}
    By \Cref{prop:PE1} the matrix $\Gamma_{L,1}(\overline u,\overline x)$ is positive definite. This shows \eqref{eq:kernels}, cf.\ Observation~7.1.8 in \cite{Horn2012}. In particular,
    \begin{align*}
        \rank \Gamma_{L,K}(\overline{u},\overline{y}) &= \rank \mathcal S_{L,K} = Lm + \rank \mathcal O_{K-1}\\
        & = Lm + \n(\B).
    \end{align*}

    We show the implication (1) to (2). Let $\col(u,y)\in\B\cap H^L(\I,\R^{m+p})$ with state $x$. Then \eqref{eq:SLambda} holds. Therefore, with \eqref{eq:images} the function $\Lambda_{L,K}(u,y)$ maps pointwise a.e.\ into $\im \Gamma_{L,K}(\overline u,\overline y)$ and, thus, $g := \Gamma_{L,K}(\overline u,\overline y)^\dagger  \Lambda_{L,K}(u,y)$ satisfies \eqref{eq:FLrep}.

    We show the converse implication. Assume that \eqref{eq:FLrep} holds. We consider a input-output representation~\eqref{eq:io} of $\B$ with polynomial matrices $Q$ and $P$. It is no restriction to assume that the degree of $Q$ and $P$ is bounded by $\lag(\B)$, i.e.\ $Q(s)=\sum_{k=0}^{\lag(\B)} Q_k s^k$ and $P(s)=\sum_{k=0}^{\lag(\B)} P_k s^k$. Define
    \begin{align*}
        \widetilde Q &:= \begin{bmatrix}
            Q_0 & \dots & Q_{\lag(B)} & 0_{p\times m(L-\lag(\B)-1)}\end{bmatrix}\;,\\
        \widetilde P &:= \begin{bmatrix} P_0 & \dots & P_{\lag(\B)} & 0_{p\times p(K-\lag(\B)-1)}
        \end{bmatrix}\; .
    \end{align*}
   Since $\col(\overline{u},\overline{y})\in\B$, it holds that 
    \begin{equation*}
        \begin{bmatrix}
            \widetilde Q & \widetilde P
        \end{bmatrix} \Lambda_{L,K}(\overline u,\overline x)=0 
    \end{equation*}
    and, consequently, 
    \begin{equation*}
        \begin{bmatrix}
            \widetilde Q & \widetilde P
        \end{bmatrix} \Gamma_{L,K}(\overline u,\overline x)=0\; .
    \end{equation*}
    Therefore,
    \begin{equation*}
        \begin{bmatrix}
            \widetilde Q & \widetilde P
        \end{bmatrix} \Lambda_{L,K}(u,y) = \begin{bmatrix}
            \widetilde Q & \widetilde P
        \end{bmatrix} \Gamma_{L,K}(\overline u,\overline x) g =0,
    \end{equation*}
    that is~\eqref{eq:io} holds and $\col(u,y)\in\B$.
\end{proof}

\begin{remark}
    Instead of using the data matrix $\Gamma_{L,K}(\overline{u},\overline{y})$, any other matrix with the same image is suitable in the description of trajectories~\eqref{eq:FLrep}. One advantageous approach, especially from a numerical perspective, is to utilize the \emph{reduced singular value decomposition} of $\Gamma_{L}(\overline{u},\overline{y})$, i.e.  
   \begin{equation*}
    \Gamma_{L}(\overline{u},\overline{y})=U_1 \Sigma_1 V_1^\top,
    \end{equation*}
    with $\Sigma_1$ nonsingular of dimension equal to the rank of $\Gamma_{L}(\overline{u},\overline{y})$. Observe that the columns of $U_1$ form an orthogonal basis for $\im\Gamma_{L}(\overline{u},\overline{y})$.
\end{remark}

In the absence of a feedthrough term (i.e., $D=0$ in the input-state-output representation), the constraint $K\leq L$ in Theorem~\ref{thm:FL} can be relaxed to $K\leq L+1$. If, in addition, the state is directly observable (i.e.\ $C=I_n$), we get the following statement.
\begin{corollary}
    \label{cor:FL}
    Suppose 
    \begin{equation}
        \label{eq:sys_is}
        \B=\Big\{\col(u,x)\in L^2(\I,\R^{m+n})\,\Big|\, \tfrac{\diff}{\diff} x = Ax + Bu\Big\}
    \end{equation}
    is controllable. Let $\col(\overline u,\overline x)\in\B$ such that $\overline u$ is persistently exciting of order $2+n$. Consider the partition
    \begin{equation*}
        \begin{bmatrix}
            \Gamma_u\\
            \Gamma_x\\
            \Gamma_{x^{(1)}}
        \end{bmatrix} = \Gamma_{1,2}(\overline{u},\overline{x})
    \end{equation*}
    with $\Gamma_u\in\R^{m\times (m+2n)}$, $\Gamma_{x}, \Gamma_{x^{(1)}}\in\R^{n\times(m+2n)}$. Then, for $u\in L^2(\I,\R^m)$ and $x\in H^1(\I,\R^n)$ the following statements are equivalent:
    \begin{itemize}
        \item[(1)] $\col(u,x)\in\B$;
        \item[(2)] There exists $g\in L^2(\I, \R^{m+2n})$ such that
        \begin{equation}
            \label{eq:FLrep_is}
            u = \Gamma_u g,\quad x = \Gamma_x g,\quad \tfrac{\diff}{\diff t}\left(\Gamma_x g\right)=\Gamma_{x^{(1)}} g.
        \end{equation}
    \end{itemize}
    Moreover, $\rank \Gamma_{1,2}(\overline{u},\overline{x}) = m + n$.
\end{corollary}

Corollary~\ref{cor:FL} allows for an complete description of $\B$ based only on sufficiently informative data, without know\-ledge of the system matrices $A$ and $B$. Note that, however, the verification of condition~\eqref{eq:FLrep} involves solving a system of linear equations. The solution of a system of linear differential equations (with time-varying coefficients) arises also in the version of the fundamental lemma in \cite{Mueller22}, see Theorem 2 therein.

\subsection{System identification}
\label{subsec:ident}
    The data matrix, as applied in the fundamental lemma, enables the reconstruction of behavioral representations. Suppose that the assumptions of \Cref{cor:FL} hold. The representation~\eqref{eq:FLrep_is} is equivalent to
    \begin{equation}
        \label{eq:RM}
        \widetilde R\tderb \col(u,x) = \widetilde M\tderb g,
    \end{equation}
    where $\widetilde R$ and $\widetilde M$ are polynomial matrices given by
    \begin{equation}
        \widetilde R(s)=\begin{bmatrix}
            0 & 0\\
            I_m & 0\\
            0 & I_n
        \end{bmatrix},\ \widetilde M(s)=\begin{bmatrix}
            \Gamma_{x^{(1)}} - s\Gamma_x\\ \Gamma_u\\\Gamma_x
        \end{bmatrix}.
    \end{equation}
    Observe, that $g$ serves as a latent variable in the representation~\eqref{eq:RM}. 

    We are going to eliminate the latent variable $g$, cf.\ Theorem 6.2.6.\ in \cite{yellowbook} Let
    \begin{equation}
        \begin{bmatrix}
            \widetilde{B} & \widetilde{A}
        \end{bmatrix} = \Gamma_{x^{(1)}} \begin{bmatrix}
            \Gamma_u\\\Gamma_x
        \end{bmatrix}^\dagger, \quad \widetilde{B}\in\R^{n\times m},\ \widetilde{A}\in\R^{n\times n}\; .
    \end{equation}
    Note that 
    \begin{equation*}
        \rank \Gamma_{1,2}(\overline{u},\overline{x}) = \rank \Gamma_{1,1}(\overline{u},\overline{x}) = m+n,       
    \end{equation*}
    cf.\ \Cref{prop:PE1} and \Cref{cor:FL}, and $\Gamma_{1,1}(\overline{u},\overline{x})$ is a submatrix of $\Gamma_{1,2}(\overline{u},\overline{x})$. Therefore, the rows of $\Gamma_{x^{(1)}}$ are linearly dependent on those of $\begin{bmatrix}
        \Gamma_u^\top & \Gamma_x^\top
    \end{bmatrix}{}^\top$.
    As a consequence, multiplication with the unimodular matrix $U$,
    \begin{equation}
        U(s) = \begin{bmatrix}
            I_n & -\widetilde{B} & (sI_n-\widetilde{A})\\ 0&I_m&0\\ 0&0 &I_n
        \end{bmatrix},
    \end{equation}
    yields
    \begin{equation}
        \label{eq:uni}
        U(s)\begin{bmatrix}\widetilde{R}(s) & \widetilde{M}(s)\end{bmatrix} = \begin{bmatrix}
            -\widetilde B & (sI_n-\widetilde{A}) & 0\\
            I_m & 0 & \Gamma_u\\
            0&I_n & \Gamma_x
        \end{bmatrix}.
    \end{equation}
    Finally, using the first $n$ rows in $U(s)\widetilde{R}(s)$, a kernel representation~\eqref{eq:ker} of $\B$ is obtained,
    \begin{equation}
        \label{}
        R(s) = \begin{bmatrix}
            -\widetilde{B} & (s I_n-\widetilde{A})
        \end{bmatrix}
     \end{equation}
    It is not difficult to see that $\widetilde{A}$ and $\widetilde{B}$ (together with $C=I_n$, $D=0$) are suitable matrices for the input-state-output~model of~\eqref{eq:sys}.

\subsection{Expansion-based formulation}
Employing the polynomial lift, see \Cref{subsec:polylift}, we obtain the following two corollaries of \Cref{thm:FL} and \Cref{cor:FL}, respectively.
\begin{corollary}
    \label{cor:FL_poly}
    Let the assumption of \Cref{thm:FL} hold. Consider the partition
    \begin{equation*}
    \begin{bmatrix}
        \Gamma_u\\
        \Gamma_{u^{(1)}}\\
        \vdots\\
        \Gamma_{u^{(L-1)}}\\
        \Gamma_y\\
        \Gamma_{y^{(1)}}\\
        \vdots\\
        \Gamma_{y^{(L-1)}}
    \end{bmatrix} = \Gamma_{L,K}(\overline{u},\overline{y}),     
    \end{equation*}
    where $\Gamma_{u^{(j)}}\in\R^{m\times (Lm+Kp)}$, $j=0,\dots,L-1$, and $\Gamma_{y^{(k)}}\in\R^{p\times (Lm+Kp)}$, $k=0,\dots,K-1$. For $\col(u,y)\in H^{L-1}(\I,\R^q)$ with $\hat u =\Pi u$, $\hat y = \Pi y$ the following statements are equivalent:
    \begin{itemize}
        \item[(1)] $\col(u,y)\in\B$;
        \item[(2)] There exists $\hat g\in \ell^2(\N, \R^{Lm+Kp})$ such that
        \begin{equation}
            \label{eq:FLrep_poly}
            \begin{split}
                \hat u &= \Gamma_{u^{(0)}} \hat g,\\
                \hat y&= \Gamma_{y^{(0)}} \hat g,\\
                \mathcal D(\Gamma_{u^{(j-1)}}\hat g) &= \Gamma_{u^{(j)}}\hat g,\quad j=1,\dots,L-1,\\
                \mathcal D(\Gamma_{y^{(k-1)}}\hat g) &= \Gamma_{y^{(k)}}\hat g,\quad k=1,\dots,K-1,
            \end{split}
        \end{equation}
        where $\mathcal D$ is defined as in \eqref{eq:infMat}.
    \end{itemize}
    In this case the $k$-th derivative of $u$ is given by
    $u^{(k)} = \sum_{i\in\N} (\Gamma_{u^{(k)}} \hat g_i) \pi_i$; similarly for derivatives of $y$.
\end{corollary}

\begin{corollary}
    \label{cor:FL_poly_is}
    Let the assumption of \Cref{cor:FL} hold. Let 
    For $u \in L^2(\I,\R^m)$ and $x\in H^{1}(\I,\R^n)$ with $\hat u =\Pi u$, $\hat x = \Pi x$ the following statements are equivalent:
    \begin{itemize}
        \item[(1)] $\col(u,y)\in\B$;
        \item[(2)] There exists $\hat g\in \ell^2(\N, \R^{m+2n})$ such that
        \begin{equation}
            \label{eq:FLrep_poly_is}
            \hat u = \Gamma_u \hat g,\quad \hat x = \Gamma_x \hat g,\quad \mathcal D\left(\Gamma_x \hat g\right)=\Gamma_{x^{(1)}} \hat g,
        \end{equation}
        where $\mathcal D$ is  defined as in \eqref{eq:infMat}.
    \end{itemize}
\end{corollary}
Conditions \eqref{eq:FLrep_poly} and \eqref{eq:FLrep_poly_is} are formulated in terms of infinite series, meaning that each coefficients $\hat g_i$ for $i\in\N$ must satisfy specific linear equations. This complicates numerical computations. Limiting considerations on polynomial trajectories, i.e.\ $\col(u,y)\in\B\cap \im P_N$, this infinite equation system is equivalently reduced to a finite one, assuming $\hat g_i=0$ for all $i\geq N$.

\section{Data-driven optimal control}  \label{sec:dataOptCon}
Finally, utilizing the approximation result of \Cref{sec:LQR_approx} in conjunction with the fundamental lemma, we propose a data-driven approach for optimal control of input-output systems (see also \cite{ChuR23} for a recent application of orthogonal bases of functions in iteratively  solving finite-length continuous-time tracking problems).

\subsection{Data-driven formulatation}
Let the assumptions of \Cref{thm:FL} and \Cref{cor:FL_poly} (with $K=L=\lag(\B)+1$) hold. We consider the optimization problem
\begin{subequations}
\label{eq:docp}
\begin{align}
\label{eq:docp1}
    \operatorname*{minimize}_{\hat g\in \ell^2(\N, \R^{Lq})}\ \sum_{i<N} \bigl(\norm{\Gamma_{y^{(0)}}\hat g_i}_2^2 + \norm{\Gamma_{u^{(\lag(\B))}} \hat g_i}^2_2\bigr)\norm{\pi_i}^2
\end{align}
subject to
\begin{align}
    \label{eq:docp3}
    \hat g_i&=0,\qquad i\geq N,\\
    \label{eq:docp2}
    \mathcal D(\Gamma_{u^{(k-1)}}\hat g) &= \Gamma_{u^{(k)}}\hat g,\\
    \label{eq:docp22}
    \mathcal D(\Gamma_{u^{(k-1)}}\hat g) &= \Gamma_{u^{(k)}}\hat g,\qquad k=1,\dots,\lag(\B),\\
    \label{eq:docp4}
    \xi^0 &= \sum_{i<N} (-1)^i\begin{bmatrix}
        \Gamma_{u^{(0)}}\\
        \Gamma_{y^{(0)}}\\
        \vdots\\
        \Gamma_{u^{(\lag(\B)-1)}}\\
        \Gamma_{y^{(\lag(\B)-1)}}
    \end{bmatrix}\hat g_i.
\end{align}
\end{subequations}
Note, that the relationship between optimization problems~\eqref{eq:docp} and~\eqref{eq:LQRapprox} is established by
\begin{equation}
    \label{eq:uygrel}
    \begin{split}
    \Pi u= \hat u =\Gamma_{u^{(0)}}\hat g,&\qquad \Pi y= \hat y=\Gamma_{y^{(0)}}\hat g,\\
    u^{(k)} = \sum_{i<N} (\Gamma_{u^{(k)}}\hat g_i) \pi_i, &\qquad y^{(k)} = \sum_{i<N} (\Gamma_{y^{(k)}}\hat g_i) \pi_i.
    \end{split}
\end{equation}
Constraints~\eqref{eq:docp2} and \eqref{eq:docp22} ensure that $w=\col(u,y)\in\mathcal \B$, while constraint~\eqref{eq:docp3} guarantees $w\in\im P_N$. The initial condition $\Lambda_{\lag(\B)}(w)(-1)=\xi^0$ is reflected by~\eqref{eq:docp4}, where $\pi_i(-1)=(-1)^i$ is used. 

The following proposition summarizes the relationship between the polynomially restricted LQR problem~\eqref{eq:LQRapprox} and its data-driven formulation~\eqref{eq:docp}.
\begin{proposition}
    \label{prop:dd}
    Let the assumptions of \Cref{thm:FL} and \Cref{cor:FL_poly} hold. Then the polynomnially restricted LQR problem~\eqref{eq:LQRapprox} and the data-driven LQR problem~\eqref{eq:docp} are equivalent in the sense that $w=\col(u,y)$ solves~\eqref{eq:LQRapprox} if and only if $\hat g$ is a solution to \eqref{eq:docp} such that \eqref{eq:uygrel} holds. In particular, their optimal values coincide.
\end{proposition}
\begin{proof}
    Via the relationship~\eqref{eq:uygrel} the target function in \eqref{eq:LQRapprox} can be equivalently rewritten into that in \eqref{eq:docp1}. Further, \Cref{cor:FL_poly} directly yields the equivalence of LQR problem~\eqref{eq:LQRapprox} and a modified data-driven formulation of \eqref{eq:docp}, where in the latter problem the constraint \eqref{eq:docp3} is replaced by the seemingly more restrictive constraint
    \begin{equation}
    \label{eq:docp3_mod}
        \begin{bmatrix}
            \Gamma_{u^{0}}\\ \Gamma_{y^{0}}
        \end{bmatrix} \hat g_i = 0,\qquad i\geq N.
    \end{equation}
    Note that the modified condition \eqref{eq:docp3_mod} in combination with \eqref{eq:uygrel} is equivalent to $\col(u,y)\in\im P_N$. Replacing \eqref{eq:docp3} with \eqref{eq:docp3_mod}, however, does not affect the feasibility or optimality of a trajectory $\col(u,y)$ with \eqref{eq:uygrel}. Indeed, this follows form the fact that \eqref{eq:docp3_mod} together with \eqref{eq:docp2}, \eqref{eq:docp22} implies
    \begin{equation}
        \begin{bmatrix}
            \Gamma_{u^{k}}\\ \Gamma_{y^{k}}
        \end{bmatrix} \hat g_i = 0,\qquad i\geq N,\quad k=0,\dots, \lag(\B)
    \end{equation}
    and $\hat g_i$ only appears in the modified problem, when accompanied by $\Gamma_{u^{k}}$ or $\Gamma_{y^{k}}$.
\end{proof}
The approximation result in \Cref{thm:conv_optm} yields asymptotic bounds on the optimality gap between the data-driven LQR problem~\eqref{eq:docp} and the original LQR problem~\eqref{eq:LQR}. We emphasize that  the allowed polynomial approximation order $N$ does not depend in the persistency of excitation order of the data, that is the same informative data trajectory $(\overline u, \overline y)$ can utilized for different $N$. \Cref{fig:scheme} illustrates how the various results in this paper integrate to derive a solution to the LQR problem~\eqref{eq:LQR} via the data-driven LQR formulation~\eqref{eq:docp}.

\begin{figure}[hbt]
    \centering
    \begin{tikzpicture}
    \node[draw, rectangle, text width=80pt, align=center, inner sep=5,outer sep=5] (LQR) {LQR~\eqref{eq:LQR}};
    \node[draw, rectangle, text width=80pt, inner sep=5,outer sep=5, align=center, below=40pt of LQR] (LQRpol) {polynomially restricted LQR~\eqref{eq:LQRapprox}};
    \node[draw, rectangle, text width=80pt, align=center, inner sep=5,outer sep=5, below=50pt of LQRpol] (LQRdd) {data-driven LQR~\eqref{eq:docp}};
    \draw[-{Triangle[width=18pt,length=8pt]}, line width=10pt, color=lightgray](LQRpol) -- (LQR) node[midway, right= 10pt, text width=80pt, align=center, text=black] {approximation\\properties from \\\Cref{thm:conv_optm}};
    \draw[-{Triangle[width=18pt,length=8pt]}, line width=10pt, text=black, color=lightgray](LQRdd) -- (LQRpol) node[midway, right= 10pt, text width=100pt, align=center, text=black] {``fundamental lemma''\\ \Cref{thm:FL}, \Cref{cor:FL_poly}, \Cref{prop:dd}};
    \end{tikzpicture}
    \caption{A schematic overview of the relations of the different LQR formulations.}
    \label{fig:scheme}
\end{figure}
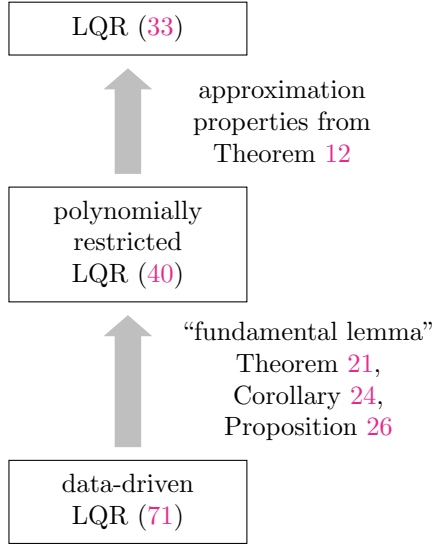

Due to \eqref{eq:docp3}, the optimization problem \eqref{eq:docp} can be rewritten as a finite-dimensional quadratic program. In this context, $\mathcal D$ in constraints~\eqref{eq:docp2} and~\eqref{eq:docp22} is replaced with some upper-left square submatrix of the infinite matrix representation of $\mathcal D$ in \eqref{eq:Dmat}.

Note, that instead of the polynomially restricted LQR problem~\eqref{eq:LQRapprox} one could likewise derive a data-driven formulation of the unrestricted LQR problem~\eqref{eq:LQR} using \Cref{cor:FL_poly}. Since the resulting problem does not include a condition like \eqref{eq:docp3}, meaning it involves infinitely many coupled equations, finding a numerical solution, however, seems intractable.

Similarly to the previous approach, consider the scenario described in \Cref{cor:FL}. The LQR problem~\eqref{eq:LQR_nofeed} in \Cref{rem:cost_nofeed}, constraint to polynomial trajectories, is equivalent to the data-driven optimization problem
\begin{subequations}
    \label{eq:docp_is}
    \begin{align}
    \label{eq:docp1_is}
        \operatorname*{minimize}_{\hat g\in \ell^2(\N, \R^{2n+m})}\ \sum_{i<N} \bigl(\norm{\Gamma_{x}\hat g_i}_2^2 + \norm{\Gamma_{u} \hat g_i}^2_2\bigr)\norm{\pi_i}^2
    \end{align}
    subject to
    \begin{align}
        \label{eq:docp3_is}
        \hat g_i&=0,\qquad i\geq N,\\
        \label{eq:docp2p}
        \mathcal D(\Gamma_{x}\hat g) &= \Gamma_{x^{(1)}}\hat g,\\
        \label{eq:docp4_is}
        x^0 &= \sum_{i<N} (-1)^i\Gamma_x\hat g_i,
    \end{align}
    \end{subequations}
    cf.\ \Cref{cor:FL_poly_is} and the proof of \Cref{prop:dd}.

\subsection{Numerical example}
We illustrate the numerical feasibility of the data-driven optimal control scheme involving the fundamental lemma consider the LQR
\begin{subequations}
\label{eq:simpleLQR}
\begin{align}
    \operatorname*{minimize}_{\col(u,x)}&\, \int_{-1}^1 \abs{u(t)}^2 + \abs{x(t)}^2\,\diff t\\
    \tfrac{\diff }{\diff t} x &= -x + u,\quad x(-1) = 1\,.
\end{align}
\end{subequations}
By Pontryagin's minimum principle the optimal trajectory $\col(u^\star,x^\star)$ to~\eqref{eq:simpleLQR} together with its co-variable $\lambda^\star$ satisfies
\begin{align*}
    \tfrac{\diff }{\diff t} x^\star &= -x^\star + u^\star,\quad x^\star(-1)=1\\
    \tfrac{\diff }{\diff t} \lambda^\star &= \lambda^\star -x^\star,\quad \lambda^\star(1)=0\\
    u^\star &=-\lambda^\star
\end{align*}
and one finds
\begin{align*}
    x^\star(t) &= \alpha\mathrm e^{-\sqrt{2}t}  \frac{(\sqrt{2}-2)\mathrm e^{2\sqrt{2}t} - (\sqrt{2}+2)\mathrm e^{2\sqrt{2}}}{ \sqrt{2}  ( \mathrm e^{2\sqrt{2}} - 1 ) }\\
    u^\star(t) &= -\alpha\mathrm e^{-\sqrt{2}t} \frac{ \mathrm e^{2\sqrt{2}t} - \mathrm e^{2\sqrt{2}} }{ \mathrm e^{2\sqrt{2}} - 1 }
\end{align*}
with a normalization constant $\alpha$ to ensure $x^\star(-1)=1$. The optimal value is $J^\star \approx 0.4125$. 

Note, that the underlying system has McMillan degree $\n(\B)=1$ and lag $\lag(\B)=1$. We consider the trajectory $\col(\overline u,\overline x)$,
\begin{equation}
    \overline u(t) = t^2, \quad \overline x(t) = t^2-2t-5\mathrm e^{-(t+1)}+2,
\end{equation}
where $\overline u$ is persistently exciting of order $3$, see \Cref{ex:pe_poly}. The smallest eigenvalue of $\Gamma_{3}(\overline u)$ is approximately $0.1729$. We numerically solve the polynomially restricted optimal control problem, cf.~\eqref{eq:LQRapprox}, in its data-driven formulation~\eqref{eq:docp_is} for different polynomial orders $N$. The resulting time-domain trajectories reconstructed from the expansion coefficients are illustrated in \Cref{fig:1}. The deviations between the optimal value $J^\star$ and the optima of the data-driven problems are presented in \Cref{tab:1}. The numerical results align with the theoretical convergence order described in \Cref{thm:conv_optm}. The \texttt{Matlab} code that produced the numerical results is available.\footnote{ \url{https://github.com/schmitzph/contDdOC}}

\begin{table}[hbt]
        \centering
        \caption{The error between the optimal value $J^\star$ and the optimal value $J^N=J(\col(u^N,x^N))$ of the data-driven LQR problem with respect to polynomial trajectories in $\im P_N$. }
        \label{tab:1}
        \begin{tabular}{l@{\hskip 20pt}l}
            \toprule
            $N$ & $J^N-J^\star$\\
            \midrule
            1 & $3.59\cdot 10^0$\\
            2 & $4.11\cdot 10^{-1}$\\
            3 & $3.36\cdot 10^{-2}$\\
            4 & $1.70\cdot 10^{-3}$\\
            5 & $4.79\cdot 10^{-5}$\\
            \bottomrule
        \end{tabular}%
        \hspace{20pt}\begin{tabular}{l@{\hskip 20pt}l}
            \toprule
           $N$ & $J^N-J^\star$\\
            \midrule
            6 & $9.58\cdot 10^{-7}$\\
            7 & $1.25\cdot 10^{-8}$\\
            8 & $1.30\cdot 10^{-10}$ \\
            9 & $9.72\cdot 10^{-13}$\\
            10 & $1.73\cdot 10^{-14}$\\
            \bottomrule
        \end{tabular}
\end{table}

\begin{figure}[hbt]
    \pgfplotsset{cycle list/Dark2}
\begin{tikzpicture}

\begin{axis}[%
width=2.55in,
height=1.3in,
at={(0in,1.65in)},
scale only axis,
xmin=-1,
xmax=1,
ymin=-0.588235305817472,
ymax=0.5,
axis x line*=bottom,
axis y line*=left,
clip=false,
legend style={draw=none, at={(1.09,0.97)},inner sep=0pt},
xtick={-1,0,1},
ylabel={input}
]

\legend{,,,,$u^\star$}
\addplot +[ultra thick]
  table[row sep=crcr]{%
-1	0.470588231394176\\
-0.9	0.417647054533593\\
-0.8	0.364705877673011\\
-0.7	0.311764700812429\\
-0.6	0.258823523951846\\
-0.5	0.205882347091264\\
-0.4	0.152941170230682\\
-0.3	0.0999999933700992\\
-0.2	0.0470588165095168\\
-0.1	-0.00588236035106553\\
0	-0.0588235372116479\\
0.1	-0.11176471407223\\
0.2	-0.164705890932813\\
0.3	-0.217647067793395\\
0.4	-0.270588244653977\\
0.5	-0.32352942151456\\
0.6	-0.376470598375142\\
0.7	-0.429411775235724\\
0.8	-0.482352952096307\\
0.9	-0.535294128956889\\
1	-0.588235305817472\\
} node[pos=0.75, below] {$u^2$};

\addplot +[ultra thick]
  table[row sep=crcr]{%
-1	-0.0601503776227316\\
-0.9	-0.103007520374129\\
-0.8	-0.139849625505443\\
-0.7	-0.170676693016674\\
-0.6	-0.195488722907821\\
-0.5	-0.214285715178885\\
-0.4	-0.227067669829866\\
-0.3	-0.233834586860763\\
-0.2	-0.234586466271577\\
-0.1	-0.229323308062307\\
0	-0.218045112232954\\
0.1	-0.200751878783518\\
0.2	-0.177443607713998\\
0.3	-0.148120299024395\\
0.4	-0.112781952714709\\
0.5	-0.0714285687849389\\
0.6	-0.0240601472350857\\
0.7	0.0293233119348509\\
0.8	0.0887218087248709\\
0.9	0.154135343134974\\
1	0.225563915165161\\
} node[pos=1.0, right, yshift=3pt] {$u^3$};

\addplot +[ultra thick]
  table[row sep=crcr]{%
-1	-0.320161212764616\\
-0.9	-0.311734845011444\\
-0.8	-0.298154259874701\\
-0.7	-0.280201842911808\\
-0.6	-0.258659979680187\\
-0.5	-0.234311055737258\\
-0.4	-0.207937456640443\\
-0.3	-0.180321567947163\\
-0.2	-0.152245775214839\\
-0.1	-0.124492464000893\\
0	-0.0978440198627457\\
0.1	-0.073082828357818\\
0.2	-0.0509912750435314\\
0.3	-0.032351745477307\\
0.4	-0.0179466252165661\\
0.5	-0.0085582998187298\\
0.6	-0.00496915484121936\\
0.7	-0.00796157584145595\\
0.8	-0.0183179483768608\\
0.9	-0.036820658004855\\
1	-0.0642520902828599\\
} node[pos=1.0, right] {$u^4$};

\addplot +[ultra thick]
  table[row sep=crcr]{%
-1	-0.395356624518486\\
-0.9	-0.353691840192539\\
-0.8	-0.313353513912976\\
-0.7	-0.274959200544653\\
-0.6	-0.239016924553944\\
-0.5	-0.205925180008734\\
-0.4	-0.17597293057843\\
-0.3	-0.149339609533952\\
-0.2	-0.126095119747735\\
-0.1	-0.106199833693731\\
0	-0.0895045934474095\\
0.1	-0.0757507106857538\\
0.2	-0.0645699666872642\\
0.3	-0.0554846123319566\\
0.4	-0.0479073681013633\\
0.5	-0.0411414240785323\\
0.6	-0.0343804399480275\\
0.7	-0.026708544995929\\
0.8	-0.0171003381098328\\
0.9	-0.00442088777885069\\
1	0.0125742679063892\\
} node[pos=1.0, right, yshift=4pt] {$u^5$};

\addplot +[dashed, ultra thick, color=black]
  table[row sep=crcr]{%
-1	-0.412519252644955\\
-0.9	-0.357707227505985\\
-0.8	-0.310061278443274\\
-0.7	-0.268626897218088\\
-0.6	-0.23257401413879\\
-0.5	-0.201180368979563\\
-0.4	-0.173817041684499\\
-0.3	-0.149935852988495\\
-0.2	-0.129058382547126\\
-0.1	-0.110766384571772\\
0	-0.0946934089630278\\
0.1	-0.0805174600855589\\
0.2	-0.0679545461150535\\
0.3	-0.0567529897290267\\
0.4	-0.0466883861655329\\
0.5	-0.0375591076427789\\
0.6	-0.0291822640780802\\
0.7	-0.0213900391858054\\
0.8	-0.014026328554057\\
0.9	-0.00694361234948237\\
1	0\\
};
\end{axis}

\begin{axis}[%
width=2.5in,
height=1.3in,
at={(0in,0.0in)},
scale only axis,
xmin=-1,
xmax=1,
ymin=-0.2,
ymax=1,
axis background/.style={fill=white},
axis x line*=bottom,
axis y line*=left,
clip=false,
legend style={draw=none, at={(1.09,0.97)},inner sep=0pt},
xtick={-1,0,1},
xlabel={time},
ylabel={state}
]
\legend{,,,,$x^\star$}

\addplot +[ultra thick]
  table[row sep=crcr]{%
-1	1\\
-0.9	0.947058823139418\\
-0.8	0.894117646278835\\
-0.7	0.841176469418253\\
-0.6	0.78823529255767\\
-0.5	0.735294115697088\\
-0.4	0.682352938836506\\
-0.3	0.629411761975923\\
-0.2	0.576470585115341\\
-0.1	0.523529408254758\\
0	0.470588231394176\\
0.1	0.417647054533594\\
0.2	0.364705877673011\\
0.3	0.311764700812429\\
0.4	0.258823523951846\\
0.5	0.205882347091264\\
0.6	0.152941170230682\\
0.7	0.0999999933700992\\
0.8	0.0470588165095168\\
0.9	-0.0058823603510656\\
1	-0.058823537211648\\
}  node[pos=0.5,above] {$x^2$};

\addplot +[ultra thick]
  table[row sep=crcr]{%
-1	1\\
-0.9	0.896992481047769\\
-0.8	0.799999999715621\\
-0.7	0.709022556003556\\
-0.6	0.624060149911575\\
-0.5	0.545112781439677\\
-0.4	0.472180450587862\\
-0.3	0.405263157356131\\
-0.2	0.344360901744484\\
-0.1	0.289473683752919\\
0	0.240601503381438\\
0.1	0.197744360630041\\
0.2	0.160902255498726\\
0.3	0.130075187987496\\
0.4	0.105263158096348\\
0.5	0.0864661658252843\\
0.6	0.0736842111743037\\
0.7	0.0669172941434064\\
0.8	0.0661654147325925\\
0.9	0.0714285729418621\\
1	0.082706768771215\\
}  node[pos=0.8, below] {$x^3$};

\addplot +[ultra thick]
  table[row sep=crcr]{%
-1	1\\
-0.9	0.874733710388237\\
-0.8	0.762445493734257\\
-0.7	0.662352964480639\\
-0.6	0.573673737069961\\
-0.5	0.495625425944802\\
-0.4	0.427425645547742\\
-0.3	0.368292010321358\\
-0.2	0.31744213470823\\
-0.1	0.274093633150937\\
0	0.237464120092057\\
0.1	0.206771209974169\\
0.2	0.181232517239853\\
0.3	0.160065656331686\\
0.4	0.142488241692247\\
0.5	0.127717887764116\\
0.6	0.114972208989871\\
0.7	0.103468819812091\\
0.8	0.0924253346733544\\
0.9	0.0810593680162405\\
1	0.0685885342833281\\
}  node[pos=1.0,right, yshift=8pt] {$x^4$};

\addplot +[ultra thick]
  table[row sep=crcr]{%
-1	1\\
-0.9	0.869237564319172\\
-0.8	0.754824911081345\\
-0.7	0.655049181436823\\
-0.6	0.568307046934393\\
-0.5	0.493104709521327\\
-0.4	0.42805790154338\\
-0.3	0.371891885744792\\
-0.2	0.323441455268286\\
-0.1	0.281650933655071\\
0	0.245574174844837\\
0.1	0.214374563175761\\
0.2	0.187325013384502\\
0.3	0.163807970606203\\
0.4	0.143315410374493\\
0.5	0.125448838621483\\
0.6	0.109919291677769\\
0.7	0.096547336272431\\
0.8	0.0852630695330323\\
0.9	0.0761061189856212\\
1	0.0692256425547293\\
}  node[pos=1.0, right, yshift=-1pt] {$x^5$};

\addplot +[dashed, ultra thick, color=black]
  table[row sep=crcr]{%
-1	1\\
-0.9	0.868293441702347\\
-0.8	0.753981714655767\\
-0.7	0.654774771387259\\
-0.6	0.568685163927957\\
-0.5	0.493988228561308\\
-0.4	0.429187535021608\\
-0.3	0.372984907974415\\
-0.2	0.32425442020849\\
-0.1	0.282019836535637\\
0	0.245435056524061\\
0.1	0.213767164267577\\
0.2	0.186381745620665\\
0.3	0.162730178754317\\
0.4	0.142338643419866\\
0.5	0.124798628739469\\
0.6	0.109758749362447\\
0.7	0.0969177060376072\\
0.8	0.086018249578217\\
0.9	0.0768420272975765\\
1	0.0692052086720081\\
};

\end{axis}
\end{tikzpicture}%
    \caption{The optimal trajectory $w^\star=\col(u^\star,x^\star)$ (dashed, black) and approximative optimal trajectories $w^N=\col(u^N,x^N)$ for $N=2,3,4,5$.}
    \label{fig:1}
\end{figure}
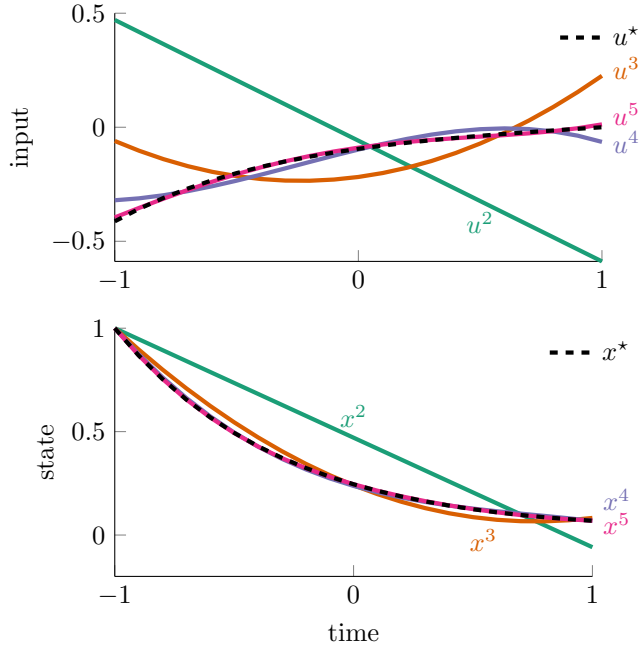

\section{Conclusions} \label{Sec:Conclusion}

We stated Gramian-based continuous-time versions of Willems et al.'s fundamental lemma in \Cref{thm:FL} and \Cref{cor:FL} in the case of input-output and input-state measurements, respectively. 
Then, we applied the derived results to the data-driven simulation problem in \Cref{cor:FL_poly,cor:FL_poly_is}. 

The evaluation of the performance of our approach in the case of noisy data is of pressing importance. The extension of our approach to the nonlinear case at least for specific classes of systems is also a matter of pressing research, especially in the light of recent results in the discrete (see \cite{Berberich21,Alsalti2023,strasser2024safedmd,molodchyk2024,lazar2023basis}) and continuous-time case~\cite{StraScha24}.


\end{document}